\documentclass[11pt]{article}

\usepackage{amsmath}
\usepackage{amsthm}
\usepackage{amssymb}
\usepackage{latexsym}
\usepackage{textcomp}
\usepackage[T1]{fontenc}
\usepackage[sc]{mathpazo}
\usepackage{titlefoot}
\usepackage{titlesec}
\usepackage[small,it]{caption}
\usepackage{pstricks}
\usepackage{graphicx, pst-plot, pst-node, pst-text, pst-tree}

\definecolor{lightgray}{rgb}{0.9, 0.9, 0.9}
\definecolor{darkgray}{rgb}{0.7, 0.7, 0.7}
\definecolor{darkblue}{rgb}{0, 0, .4}

\usepackage[bookmarks]{hyperref}
\hypersetup{
	colorlinks=true,
	linkcolor=darkblue,
	anchorcolor=darkblue,
	citecolor=darkblue,
	urlcolor=darkblue,
	pdfpagemode=UseThumbs,
	pdftitle={Grid classes and partial well-order},
	pdfsubject={Combinatorics},
	pdfauthor={Robert Brignall},
	pdfkeywords={todo}
}
\theoremstyle{plain}
\newtheorem{theorem}{Theorem}[section]
\newtheorem{corollary}[theorem]{Corollary}
\newtheorem{lemma}[theorem]{Lemma}
\newtheorem{proposition}[theorem]{Proposition}

\newtheorem{question}[theorem]{Question}
\theoremstyle{definition}

\theoremstyle{remark}

\newcounter{todocounter}

\makeatletter
\newfont{\footsc}{cmcsc10 at 8truept}
\newfont{\footbf}{cmbx10 at 8truept}
\newfont{\footrm}{cmr10 at 10truept}
\title{Grid Classes and Partial Well Order}
\author{Robert Brignall\footnote{This research was conducted while the author was being supported by the Heilbronn Institute for Mathematical Research.}\\[-3pt]
\small Department of Mathematics and Statistics\\[-3pt]
\small The Open University\\[-3pt]
\small Milton Keynes, UK\\[-3pt]
\small \texttt{r.brignall@open.ac.uk}\\[-3pt]
\small \texttt{http://users.mct.open.ac.uk/rb8599}}
\date{9th August 2011}

\begin{document}
\maketitle

\newcommand{\OEISlink}[1]{{#1}}
\newcommand{\OEISref}{{OEIS}~\cite{OEIS}}
\newcommand{\OEIS}[1]{(Sequence \OEISlink{#1} in the \OEISref.)}

\newcommand{\sub}[1]{\mathrm{Cl}(#1)}
\newcommand{\propsub}[1]{\mathrm{Pcl}(#1)}
\newcommand{\cl}[1]{\mathrm{Cl}(#1)}
\newcommand{\pcl}[1]{\mathrm{Pcl}(#1)}
\newcommand{\av}[1]{\mathrm{Av}(#1)}
\newcommand{\rect}{\mathrm{rect}}
\newcommand{\grid}{\mathrm{Grid}}
\newcommand{\gr}{\mathrm{gr}}
\newcommand{\A}{\mathcal{A}}
\newcommand{\B}{\mathcal{B}}
\newcommand{\C}{\mathcal{C}}
\newcommand{\D}{\mathcal{D}}
\newcommand{\F}{\mathcal{F}}
\renewcommand{\O}{\mathcal{O}}
\newcommand{\M}{\mathcal{M}}
\newcommand{\N}{\mathcal{N}}
\newcommand{\mm}{\mathrm{mm}}
\newcommand{\si}{\mathrm{Si}}


\begin{abstract}
We prove necessary and sufficient conditions on a family of (generalised) gridding matrices to determine when the corresponding permutation classes are partially well-ordered. One direction requires an application of Higman's Theorem and relies on there being only finitely many simple permutations in the only non-monotone cell of each component of the matrix. The other direction is proved by a more general result that allows the construction of infinite antichains in any grid class of a matrix whose graph has a component containing two or more non-monotone-griddable cells. The construction uses a generalisation of pin sequences to grid classes, together with a number of symmetry operations on the rows and columns of a gridding.
\end{abstract}

\section{Introduction}

A partial order is \emph{partially well-ordered} if it contains neither an infinite \emph{antichain} (a set of pairwise incomparable elements) nor an infinite descending chain. In the study of classes of combinatorial structures this latter condition is trivially satified, thus such a class is partially well-ordered if and only if it contains no infinite antichain. For many combinatorial structures we have only a quasi-ordering rather than a partial ordering, and in this case we call such a class \emph{well quasi-ordered} when it contains no infinite antichain. Celebrated results affirming well quasi-ordering in different contexts range from Kruskal's Tree Theorem~\cite{kruskal:well-quasi-ordering:} to the Robertson--Seymour Theorem~\cite{robertson:graph-minors-xx:} for minor-closed classes of graphs, but there are many known examples of quasi-orders that are not well quasi-ordered, such as hereditary properties of graphs. Higman's Theorem (reproduced here in Section~\ref{sec-pwo-grid-classes}) is one of the few general tools available to prove that a given quasi-order is well quasi-ordered, but attention has been given more recently to develop a general theory of infinite antichains --- see, for example, Gustedt~\cite{gustedt:finiteness-theo:} and Cherlin and Latka~\cite{cherlin:minimal-anticha:}.

In this paper we are concerned with permutations, though there is no particular reason why parts of these results cannot be extended to other structures. A sequence $a_1,\ldots, a_n$ of length $n$ of distinct real numbers is said to be \emph{order isomorphic} to another sequence $b_1,\cdots, b_n$ if, for all $i,j\in [n]=\{1,2,\ldots,n\}$, $a_i<a_j$ if and only if $b_i<b_j$. In this way every sequence of real numbers of length $n$ is order isomorphic to some permutation $\pi$ of length $n$: $a_i<a_j$ if and only if $\pi(i)<\pi(j)$. This order isomorphism induces the \emph{containment} ordering on permutations: we say that a permutation $\alpha$ is \emph{contained} in $\pi$, $\alpha\leq\pi$, if there is some subsequence of $\pi$ order isomorphic to $\alpha$. Such a subsequence of $\pi$ is called a \emph{copy} of $\alpha$ in $\pi$. Conversely, if $\pi$ does not contain the permutation $\beta$, then $\pi$ is said to \emph{avoid} $\beta$. For example, $\pi=918572346$ contains $51342$ because of the subsequence $91572$ ($=\pi(1)\pi(2)\pi(4)\pi(5)\pi(6)$), but avoids $3142$.

The containment ordering on permutations defines a partial order on the set of all permutations. A \emph{permutation class} is a set of permutations closed downward in this partial order, i.e.\ if $\pi$ is a permutation in the class $\C$ and $\alpha\leq \pi$, then $\alpha\in \C$. These classes have received a lot of attention in recent years, and the question of partial well-order has played a central role: there is a vast library of infinite antichains (see, in particular Murphy's thesis~\cite{murphy:restricted-perm:}), while Higman's Theorem has been applied in the other direction by Atkinson, Murphy and Ru\v{s}kuc~\cite{atkinson:partially-well-:} and Albert and Atkinson~\cite{albert:simple-permutat:}.

The traditional description of a class $\C$ is by the unique antichain $B$ that forms its \emph{basis}: we write $\C=\av{B}$ to mean $\C=\{\pi:\beta\not\leq\pi \textrm{ for all }\beta\in B\}$. However, in recent years a new description of permutation classes has arisen, namely ``grid classes'' of matrices whose entries are themselves permutation classes --- for formal definitions see Section~\ref{sec-definitions}. These have played a role in the development of the ``Fibonacci'' and ``Vatter'' dichotomies~\cite{huczynska:grid-classes-an:,vatter:small-permutat:}, providing a complete answer to the possible growth rates\footnote{All permutation classes have an \emph{upper growth rate}, $\overline{\gr}(\C)=\limsup_{n\rightarrow\infty}\sqrt[n]{|\C_n|}$ (see~\cite{marcus:excluded-permut:}) where $\C_n$ is the set of permutations in $\C$ of length $n$, but it is still not known in general whether the true \emph{growth rate}, $\lim_{n\rightarrow\infty}\sqrt[n]{|\C_n|}$, exists for all permutation classes.} of permutation classes below $\kappa\approx 2.20557$, and in particular proving that there are only countably many classes below this growth rate. 
Grid classes are now being intensely studied in topics ranging from direct enumeration~\cite{aab:2143-4231-enum:} to connections with geometry~\cite{albert:geometric-grid-classes:}, but of particular relevance to this paper is Murphy and Vatter~\cite{murphy:profile-classes:} where grid classes and partial well-order first met, and subsequent work in Waton's thesis~\cite{waton:on-permutation-cl:}, later published in an article with Vatter~\cite{vatter:on-partial-well-order:}. In this paper, we will prove the following:

\begin{theorem}\label{thm-fin-simple-grid} Let $\M$ be a gridding matrix whose non-empty entries are monotone classes, or non-monotone-griddable classes containing only finitely many simple permutations. Then the permutation class $\grid(\M)$ is partially well-ordered if and only if the graph of $\M$ is a forest, and at most one cell in each component is not monotone.\end{theorem}

The bulk of the work in proving Theorem~\ref{thm-fin-simple-grid} is in showing:

\begin{theorem}\label{thm-not-pwo}Let $\M$ be a gridding matrix where every non-empty cell is an infinite permutation class. Then $\grid(\M)$ is not partially well-ordered if $\M$ has a cycle, or a component containing two or more cells that are not monotone griddable.\end{theorem}

After introducing the necessary definitions in Section~\ref{sec-definitions}, Section~\ref{sec-pwo-grid-classes} presents Higman's theorem and completes the proof of the right-to-left direction of Theorem~\ref{thm-fin-simple-grid}; the remainder of the paper is devoted to proving Theorem~\ref{thm-not-pwo}. In Section~\ref{sec-symmetry} we introduce a number of symmetries of griddings which reduces the number of classes that have to be considered. In Section~\ref{sec-family} we introduce a family of grid matrices and show that they are the only ones we need to consider, and in Section~\ref{sec-grid-antichains} we show that these classes are not partially well-ordered by constructing antichains that lie in them which satisfy the additional properties required by the symmetry arguments.

\section{Definitions}\label{sec-definitions}

As has become increasingly the case in the study of permutation patterns in recent years, it will prove very useful to view permutations and order isomorphism graphically. Two sets $S$ and $T$ of points in the plane are said to be order isomorphic if we can stretch and shrink the axes for the set $S$ to map the points of $S$ bijectively onto the points of $T$, i.e.\ if there are strictly increasing functions $f,g:\mathbb{R}\rightarrow\mathbb{R}$ such that $\{(f(s_1),g(s_2)) : (s_1,s_2)\in S\}=T$. Note that this forms an equivalence relation since the inverse of a strictly increasing function is also strictly increasing. The {\it plot\/} of the permutation $\pi$ is the point set $\{(i,\pi(i))\}$, and every finite point set in the plane in which no two points share a coordinate (often called a {\it generic\/} or {\it noncorectilinear\/} set) is order isomorphic to the plot of a unique permutation (see Figure~\ref{fig-perm-ex} for an example). Note that, with a slight abuse of terminology, we will say that a point set is order isomorphic to a permutation.

\begin{figure}
\begin{center}
\psset{xunit=0.01in, yunit=0.01in}
\psset{linewidth=0.005in}
\begin{pspicture}(0,0)(90,90)
\psaxes[dy=10,Dy=1,dx=10,Dx=1,tickstyle=bottom,showorigin=false,labels=none](0,0)(90,90)
\pscircle*(10,60){0.04in}
\pscircle*(20,30){0.04in}
\pscircle*(30,50){0.04in}
\pscircle*(40,80){0.04in}
\pscircle*(50,20){0.04in}
\pscircle*(60,90){0.04in}
\pscircle*(70,70){0.04in}
\pscircle*(80,10){0.04in}
\pscircle*(90,40){0.04in}
\end{pspicture}
\end{center}
\caption{The plot of the permutation $\pi=635829714$.}\label{fig-perm-ex}
\end{figure}
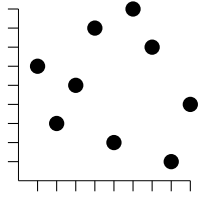

\paragraph{Inflations and Simple Permutations.} An \emph{interval} of a permutation $\pi$ corresponds to a set of contiguous indices $I = [a,b]=\{a,a+1,\ldots,b\}$ such that the set of values $\pi(I) = \{\pi(i) : i\in I\}$ is also contiguous. For example, $645=\pi(345)$ is an interval in $\pi=72645813$.

We form an \emph{inflation of $\sigma$ by the permutations $\tau_1,\ldots,\tau_k$} by replacing the entry $\sigma(i)$ with an interval order isomorphic to $\tau_i$, and denote it by $\sigma[\tau_1,\ldots,\tau_k]$. For example, $2413[21,312,1,12]=32 867 1 45$. Two special cases of inflations are the \emph{direct sum} $\tau_1\oplus\tau_2=12[\tau_1,\tau_2]$ and the \emph{skew sum} $\tau_1\ominus\tau_2=21[\tau_1,\tau_2]$. A \emph{lenient inflation} is an inflation $\sigma[\tau_1,\ldots,\tau_k]$ where we allow one or more of the $\tau_\ell$ to be empty. A class $\C$ is \emph{substitution-closed} (or, in some texts, \emph{wreath-closed}) if $\sigma[\tau_1,\ldots,\tau_k]\in \C$ for all $\sigma,\tau_1,\ldots,\tau_k\in \C$. The \emph{substitution closure} of a set $X$ is the smallest substitution-closed class containing $X$, and is denoted $\langle X\rangle$.

A \emph{simple permutation} is a permutation which has no non-trivial intervals, or equivalently a permutation which cannot be expressed as an inflation of some smaller non-singleton permutation. Conversely:

\begin{proposition}[Albert and Atkinson~\cite{albert:simple-permutat:}]\label{prop-simple-inflation} Every permutation except $1$ can be expressed as the inflation of a unique simple permutation of length at least $2$.
\end{proposition}

This proposition shows how simple permutations can be thought of as the ``building blocks'' of all other permutations, and consequently they play an important role in the study of permutation classes and have received much attention in recent years --- see~\cite{brignall:a-survey-of-sim:} for a survey. We will denote by $\si(\C)$ the set of simple permutations in the class $\C$. Note that $\si(\C)=\si(\langle\C\rangle)$, and also that $\langle\C\rangle = \langle\si(\C)\rangle$.

\paragraph{Grid Classes.} We will present here only a brief survey of the necessary results, and refer the reader to Vatter~\cite{vatter:small-permutat:} for a more complete treatment of this topic. To draw a parallel with the way we view permutations graphically, we will index matrices and grids starting from the bottom-left corner, and with the order of indices swapped. In other words, the $ij$th entry of a matrix (respectively, $ij$th cell of a grid) corresponds to the entry (cell) in column $i$ and row $j$, and an $m\times n$ matrix has $m$ columns and $n$ rows.

An $m\times n$-\emph{gridding} of a permutation $\pi$ is a collection of $m-1$ distinct vertical and $n-1$ distinct horizontal lines that divide the plot of $\pi$ into $mn$ cells. A permutation equipped with a particular $m \times n$-gridding is called an $m\times n$-\emph{gridded permutation}, and for such a gridded permutation $\pi$, $\pi^{st}$ denotes the set of points contained in the $st$th cell.

Let $\M$ be an $m\times n$ matrix where each entry is a permutation class (noting that we permit the empty class $\emptyset$): $\M$ is called a \emph{gridding matrix}. (To avoid trivialities, we will always assume that $\M$ does not have any rows or columns consisting entirely of empty cells.) An $\M$-gridding of a permutation $\pi$ is an $m\times n$ gridding of $\pi$ such that $\pi^{st}$ lies in the class $\M_{st}$ for all $s\in[m]$ and $t\in[n]$. If $\pi$ possesses an $\M$-gridding, then $\pi$ is said to be \emph{$\M$-griddable}, and equipping $\pi$ with such a gridding gives rise to an \emph{$\M$-gridded permutation}. Similarly, a permutation class $\C$ is said to be \emph{$\M$-griddable} if every $\pi\in\C$ is $\M$-griddable. The largest permutation class that is $\M$-griddable (i.e.\ the class consisting of all $\M$-griddable permutations) is called the \emph{grid class} of $\M$, and is denoted $\grid(\M)$. One special case that has received particular attention has been that of \emph{monotone grid classes}, where $\M$ has only monotone (i.e.\ the classes $\av{21}$ and $\av{12}$) or empty entries.

Now let $\C$ and $\D$ be permutation classes. We say that $\C$ is \emph{$\D$-griddable} if there is some matrix $\M$ whose entries are all subclasses of $\D$ for which $\C$ is $\M$-griddable. The following theorem gives a good characterisation of $\D$-griddability:

\begin{theorem}[Vatter~\cite{vatter:small-permutat:}]\label{thm-griddable}
A permutation class $\C$ is $\D$-griddable if and only if it does not contain arbitrarily long direct sums or skew sums of basis elements of $\D$.\end{theorem}

A particular instance of this theorem is that a permutation class is monotone griddable if and only if it does not contain arbitrarily long direct sums of $21$ or skew sums of $12$. Define the \emph{sum completion} of a permutation $\pi$ to be the permutation class $\oplus\pi=\{\alpha_1\oplus\alpha_2\oplus\cdots\oplus\alpha_k : \alpha_i\leq\pi\textrm{ for all }i\leq k\in \mathbb{N}\}$, and the \emph{skew completion} $\ominus\pi$ analogously. Thus:

\begin{corollary}\label{cor-non-monotone}A permutation class $\C$ is monotone griddable if and only if it contains neither the class $\oplus 21$ nor the class $\ominus 12$.
\end{corollary}

\paragraph{Grid classes and partial well-order.} The \emph{graph} of the gridding matrix $\M$ is the graph $G_\M$ whose vertices are the non-empty cells of $\M$, with two vertices being adjacent if they share a row or a column of $\M$ and all cells between them are empty. A \emph{component} of $\M$ is a submatrix $\M'$ of $\M$ for which $G_{\M'}$ is a connected component of $G_\M$. In determining whether grid classes are partially well-ordered, it is sufficient to look at these components individually:

\begin{proposition}[Vatter~\cite{vatter:small-permutat:}]\label{prop-component} $\grid(\M)$ is partially well-ordered if and only if $\grid(\M')$ is partially well-ordered for every connected component $\M'$ of $\M$.\end{proposition}

In the case of monotone grid classes, the connection between $G_\M$ and partial well-order is well known:

\begin{theorem}[Murphy and Vatter~\cite{murphy:profile-classes:}]\label{thm-murphy-vatter} The monotone grid class $\grid(\M)$ is partially well-ordered if and only if $G_\M$ is a forest.\end{theorem}

One direction of this theorem is proved by constructing an antichain that ``winds around'' the cells corresponding to a cycle of $G_\M$, while the other requires Higman's Theorem and has been reproved more efficiently by Waton~\cite{waton:on-permutation-cl:}. Our proof of Theorems~\ref{thm-fin-simple-grid} and~\ref{thm-not-pwo} will borrow a lot from the techniques in these two publications.

\section{Partially Well Ordered Grid Classes}\label{sec-pwo-grid-classes}

We complete one half of the proof of Theorem~\ref{thm-fin-simple-grid} by proving the following theorem.

\begin{theorem}\label{thm-grid-pwo} Let $\M$ be a gridding matrix whose entries are all permutation classes containing only finitely many simple permutations, and for which $G_\M$ is a forest and every component of $\M$ contains at most one cell labelled by a class that is not monotone. Then $\grid(\M)$ is partially well-ordered.\end{theorem}

We begin by giving a complete presentation of Higman's Theorem, which will form the backbone of the proof of Theorem~\ref{thm-grid-pwo}. We say that $(A,M)$ is an \emph{abstract algebra} if $A$ is a set of elements and $M$ a set of operations where each $\mu\in M$ has an \emph{arity}: we say that $\mu$ is a $k$-ary operation if $\mu:A^k\rightarrow A$ for some positive integer $k$. Denote the set of $k$-ary operations in $M$ by $M_k$, and suppose that $M_k$ is empty for every $k>n$ for some $n$. (Note that we will allow $0$-ary operations.) The abstract algebra $(A,M)$ is said to be \emph{minimal} if no subset $B$ of $A$ allows $(B,M)$ to be an abstract algebra.

A partial order $\leq_A$ on the set of elements $A$ is a \emph{divisibility order} on $(A,M)$ if every operation $\mu\in M_k$, $k=0,1,\ldots,n$, satisfies
\begin{itemize}
\item $a \leq_A b$ implies $\mu(\mathbf{x}, a, \mathbf{y}) \leq_A \mu(\mathbf{x}, b, \mathbf{y})$,
\item $a \leq_A \mu(\mathbf{x}, a, \mathbf{y})$,
\end{itemize}
where $\mathbf{x}$ and $\mathbf{y}$ are arbitrary sequences comprising elements of $A$ whose lengths sum to $k-1$. Furthermore, given partial orders $\leq_{M_k}$ on $M_k$, $k=0,1,\ldots,n$, we say that $\leq_A$ is \emph{compatible} with these partial orders if, for $\lambda,\mu\in M_k$,
\begin{itemize}
\item $\lambda \leq_{M_k} \mu$ implies $\lambda(\mathbf{x}) \leq_A \mu(\mathbf{x})$ for all $\mathbf{x}\in A^k$.
\end{itemize}

\begin{theorem}[Higman~\cite{higman:ordering-by-div:}]\label{thm-higman}
Suppose that $(A,M)$ is a minimal abstract algebra for which, for some $n$, the set $M_k$ of
$k$-ary operations in $M$ is partially well-ordered for each $k=0,1,\ldots,n$ and empty for $k>n$. Then $(A,M)$
is partially well-ordered under any divisibility ordering compatible with the orders of $M_k$.
\end{theorem}

Applying this to permutation classes, one type of operation that has been particularly amenable to this approach is the inflation of one permutation by others; inflating a permutation $\sigma$ of length $k$ by $\tau_1,\ldots,\tau_k$ may be thought of as a $k$-ary operation that acts on the permutations $\tau_1,\ldots,\tau_k$. It is clear both that inflation is compatible with the permutation containment ordering and that permutation containment is a divisibility ordering with respect to inflations of this type. To satisfy the conditions of Higman's Theorem, however, we cannot inflate arbitrarily large permutations. Roughly speaking, if a permutation class $\C$ is a subclass of some substitution-closed class $\D$ that can be expressed as the substitution closure of some finite set $X$, then Higman's Theorem can be applied to prove that $\D$ (and consequently $\C$) is partially well-ordered. Consequently, by Proposition~\ref{prop-simple-inflation}:

\begin{theorem}[Albert and Atkinson~\cite{albert:simple-permutat:}]Let $\C$ be a class containing only finitely many simple permutations. Then $\C$ is partially well-ordered.\end{theorem}

On the other hand, since any set $X$ satisfying $\langle X\rangle=\langle\C\rangle$ must contain every permutation in $\si(\C)$, we cannot arrange that $X$ is finite when $\C$ contains infinitely many simple permutations, and Higman's Theorem cannot be used in this way. This, however, does not mean that any class containing infinitely many simple permutations is not partially well-ordered: for example, $\grid(\av{21}\quad\av{21})$ is partially well-ordered by Theorem~\ref{thm-murphy-vatter}, but contains arbitrarily long simple permutations of the form $2\ 4\ 6\cdots 2k\ 1\ 3\ 5\cdots 2k-1$.

Let us now extend this use of Higman's Theorem to gridding matrices. We first define an order on the set of $m\times n$-gridded permutations. For $m\times n$-gridded permutations $\alpha$ and $\pi$ of lengths $k$ and $\ell$, respectively, we say that $\alpha$ is contained in $\pi$, $\alpha\leq_{mn}\pi$, if and only if there is a sequence of indices $1\leq i_1<\cdots< i_k\leq \ell$ such that $\pi(i_1)\cdots\pi(i_k)$ is order isomorphic to $\alpha$ as ungridded permutations, and for $j=1,\ldots,k$, $\pi(i_j)$ and $\alpha(j)$ lie in the same cell in the $m\times n$-griddings. Similarly, for a specific $m\times n$ gridding matrix $\M$ and $\M$-gridded permutations $\alpha$ and $\pi$, we write $\alpha\leq_\M\pi$ to mean $\alpha\leq_{mn}\pi$, but also recognising that both $\alpha$ and $\pi$ are $\M$-gridded.

Suppose that $\M$ is a gridding matrix whose graph is acyclic, and every non-empty cell of $\M$ is labelled by a monotone class, except for the $uv$th cell which is labelled by some arbitrary class $\D$. Viewing $G_\M$ as a tree rooted on the $uv$th cell, each cell other than the $uv$th is the \emph{child} of some \emph{parent} cell, i.e.\ the cell lying directly above it in the rooted tree.

Now let $\tau_1,\ldots,\tau_k$ be $\M$-gridded permutations each with at least one point in cell $uv$, and let $\sigma\in\D$ be of length $k$.  The \emph{$\M$-inflation of $\sigma$ by $\tau_1,\ldots,\tau_k$} is the $\M$-gridded permutation $\pi = \sigma[\tau_1,\ldots,\tau_k]_\M$, and is formed by first taking the inflation $\pi^{uv}=\sigma[\tau_1^{uv},\ldots,\tau_k^{uv}]$. For every other non-empty cell $st$ of $\M$, if $\M_{st}=\av{21}$ then $\pi^{st}$ is a (posibly lenient) inflation of $1\cdots k$ by $\tau_1^{st},\ldots,\tau_k^{st}$ in some order, while if $\M_{st}=\av{12}$ then $\pi^{st}$ is a (possibly lenient) inflation of $k\cdots 1$ by $\tau_1^{st},\ldots,\tau_k^{st}$ in some order. In either case, the order in which $\tau_1^{st},\ldots,\tau_k^{st}$ appears in the inflation is defined recursively in terms of its parent cell $\pi^{s't'}$ (where either $s=s'$ or $t=t'$) in the tree $G_\M$ rooted on cell $uv$: if cells $st$ and $s't'$ share a column (i.e.\ if $s=s'$) then reading from left-to-right the $\tau_i^{st}$ appear in the same order as the $\tau_i^{s't'}$. Similarly, if $t=t'$ then reading from bottom-to-top the $\tau_i^{st}$ appear in the same order as the $\tau_i^{s't'}$. For each $i\in[k]$, the positions of the points in $\tau_i^{st}$ relative to the points in $\tau_i^{s't'}$ are exactly the same as the corresponding points in the $\M$-gridded permutation $\tau_i$. Finally, for each $i\in [k]$, $\tau_i^{st}$ interacts with no other $\tau_{j}^{s't'}$, $j\neq i$: i.e.\ all the points in $\tau_i^{st}$ are above or below, and to the left or to the right of all points in each $\tau_{j}^{s't'}$. Note that we must remember the left-to-right and bottom-to-top order of $\tau_1^{st},\ldots,\tau_k^{st}$ in every non-empty cell $st$ of $\M$ even if one or more of the $\tau_i^{st}$ contains no points, so that we know the order of the cells for any subsequent descendants. See Figure~\ref{fig-gridinflation} for an illustration.

A \emph{lenient $\M$-inflation of $\sigma$ by $\tau_1,\ldots,\tau_k$} is defined in exactly the same way, except that we do not stipulate that each $\tau_i^{uv}$ be non-empty.

\begin{figure}
\begin{center}
\psset{xunit=0.012in, yunit=0.012in}
\psset{linewidth=0.005in}
\begin{pspicture}(0,0)(320,240)
\psline[linestyle=dashed](0,80)(320,80)
\psline[linestyle=dashed](0,160)(320,160)
\psline[linestyle=dashed](80,0)(80,240)
\psline[linestyle=dashed](160,0)(160,240)
\psline[linestyle=dashed](240,0)(240,240)
\psline[linestyle=dotted](0,20)(320,20)
\psline[linestyle=dotted](0,40)(320,40)
\psline[linestyle=dotted](0,60)(320,60)
\psline[linestyle=dotted](0,100)(320,100)
\psline[linestyle=dotted](0,120)(320,120)
\psline[linestyle=dotted](0,140)(320,140)
\psline[linestyle=dotted](0,180)(320,180)
\psline[linestyle=dotted](0,200)(320,200)
\psline[linestyle=dotted](0,220)(320,220)
\psline[linestyle=dotted](20,0)(20,240)
\psline[linestyle=dotted](40,0)(40,240)
\psline[linestyle=dotted](60,0)(60,240)
\psline[linestyle=dotted](100,0)(100,240)
\psline[linestyle=dotted](120,0)(120,240)
\psline[linestyle=dotted](140,0)(140,240)
\psline[linestyle=dotted](180,0)(180,240)
\psline[linestyle=dotted](200,0)(200,240)
\psline[linestyle=dotted](220,0)(220,240)
\psline[linestyle=dotted](260,0)(260,240)
\psline[linestyle=dotted](280,0)(280,240)
\psline[linestyle=dotted](300,0)(300,240)
{
\psset{fillstyle=solid,fillcolor=white,cornersize=absolute,linearc=1pt,linewidth=0.005in,linestyle=solid,linecolor=black}
  \psframe(0,140)(20,160)\psframe(20,120)(40,140)
  \psframe[fillcolor=lightgray](40,100)(60,120)\psframe(60,80)(80,100)
  \psframe(80,160)(100,180)\psframe(100,180)(120,200)
  \psframe[fillcolor=lightgray](120,200)(140,220)\psframe(140,220)(160,240)
  \psframe[fillcolor=lightgray](160,0)(180,20)\psframe[fillcolor=lightgray](160,100)(180,120)
  \psframe(180,20)(200,40)\psframe(180,140)(200,160)
  \psframe(200,40)(220,60)\psframe(200,80)(220,100)
  \psframe(220,60)(240,80)\psframe(220,120)(240,140)
  \psframe(240,80)(260,100)\psframe(240,220)(260,240)
  \psframe[fillcolor=lightgray](260,100)(280,120)\psframe[fillcolor=lightgray](260,200)(280,220)
  \psframe(280,120)(300,140)\psframe(280,180)(300,200)
  \psframe(300,140)(320,160)\psframe(300,160)(320,180)
}
{\small%
\rput[c](10,150){$\tau_2^{12}$}\rput[c](30,130){$\tau_4^{12}$}
\rput[c](50,110){$\tau_1^{12}$}\rput[c](70,90){$\tau_3^{12}$}
\rput[c](90,170){$\tau_2^{23}$}\rput[c](110,190){$\tau_4^{23}$}
\rput[c](130,210){$\tau_1^{23}$}\rput[c](150,230){$\tau_3^{23}$}
\rput[c](170,10){$\tau_1^{31}$}\rput[c](170,110){$\tau_1^{32}$}
\rput[c](190,30){$\tau_2^{31}$}\rput[c](190,150){$\tau_2^{32}$}
\rput[c](210,50){$\tau_3^{31}$}\rput[c](210,90){$\tau_3^{32}$}
\rput[c](230,70){$\tau_4^{31}$}\rput[c](230,130){$\tau_4^{32}$}
\rput[c](250,90){$\tau_3^{42}$}\rput[c](250,230){$\tau_3^{43}$}
\rput[c](270,110){$\tau_1^{42}$}\rput[c](270,210){$\tau_1^{43}$}
\rput[c](290,130){$\tau_4^{42}$}\rput[c](290,190){$\tau_4^{43}$}
\rput[c](310,150){$\tau_2^{42}$}\rput[c](310,170){$\tau_2^{43}$}}
\end{pspicture}
\end{center}
\caption{Forming the $\M$-inflation $2413[\tau_1,\tau_2,\tau_3,\tau_4]_\M$ for a $4\times 3$ gridding matrix $\M$. The highlighted cells correspond to the $\M$-gridded permutation $\tau_1$.}
\label{fig-gridinflation}
\end{figure}

\begin{proof}[Proof of Theorem~\ref{thm-grid-pwo}] By Proposition~\ref{prop-component}, we may assume that $G_\M$ consists of exactly one component. Thus $\M$ is an $m\times n$ gridding matrix such that $G_\M$ is a tree and every non-empty cell of $\M$ is labelled by a monotone class, except for the $uv$th cell which is labelled by some class $\D$ containing only finitely many simple permutations. We will also assume that $\D$ is substitution closed, as otherwise we may replace it with $\langle\D\rangle$ and prove the result for this larger class.

For each $\sigma\in\si(\D)$ of length $k$, we view an $\M$-inflation of $\sigma$ as a $k$-ary operation. We claim that $\grid(\M)$ is generated by this finite list of $\M$-inflations and all the $\M$-griddings of the singleton permutation $1$. It will then follow by Higman's Theorem~\ref{thm-higman} that $\grid(\M)$ is partially well-ordered. We will prove the result for $\M$-gridded permutations, and then the result for ungridded permutations in $\grid(\M)$ will follow by applying the homomorphism that removes the gridlines.

We proceed by induction on the length of $\M$-gridded permutations. As we already have all the $\M$-gridded permutations of length $1$, it is enough to show that any $\M$-gridded $\pi\in\grid(\M)$ with $|\pi|\geq 2$ can be expressed as an $\M$-inflation of some $\sigma\in\si(\D)$. Given one such $\pi$, suppose first that $\pi^{uv}$ contains at least two points. By Proposition~\ref{prop-simple-inflation} there exists some $\sigma\in\si(\D)$ such that $\pi^{uv}$ is an inflation of $\sigma$, i.e.\  $\pi^{uv}=\sigma[\tau_1^{uv},\ldots,\tau_k^{uv}]$, for some permutations $\tau_1^{uv},\ldots,\tau_k^{uv}$. Label each point of $\pi^{uv}$ with the symbol from $1,\ldots,k$ corresponding to which of $\tau_1^{uv},\ldots,\tau_k^{uv}$ it belongs. We now label each cell recursively, working down the tree $G_\M$ rooted at the cell $uv$. Consider a cell $st$ whose parent $rw$ has been labelled. We will label each point $p$ in $\pi^{st}$ as follows:
\begin{itemize}
\item If the child shares a column with its parent (i.e.\ $r=s$), then $p$ is assigned the same label as the rightmost point in $\pi^{rw}$ that lies to its left. If there is no point in $\pi^{rw}$ to the left of $p$, give $p$ the label of the leftmost point of $\pi^{rw}$. If there are no points in $\pi^{rw}$, label every point of $\pi^{st}$ with the label $1$.
\item If the child shares a row with its parent (i.e.\ $t=w$), then $p$ is assigned the same label as the highest point of $\pi^{rw}$ that lies below it. If there is no such point in $\pi^{rw}$, give $p$ the label of the lowest point of $\pi^{rw}$. If there are no points in $\pi^{rw}$, label every point of $\pi^{st}$ with the label $1$.
\end{itemize}
For each $i\in[k]$, now create the $\M$-gridded permutation $\tau_i$ by taking all points of $\pi$ with label $i$. It is now clear to see that $\pi=\sigma[\tau_1,\ldots,\tau_k]_\M$, as required.

This leaves the case where $\pi^{uv}$ contains a singleton or is empty. Since $|\pi|\geq 2$, either there is a cell of $\pi$ containing at least two points, or there are at least two non-empty cells. If there is a cell $\pi^{st}$ containing at least two points, label the leftmost point with the label $1$ and all other points in this cell with label $2$. Then view $G_\M$ as a tree rooted at the cell $st$ and label the points in the cells of $\pi$ recursively as described above. Using these labels, now form the gridded permutations $\tau_1$ and $\tau_2$ as before, and observe that $\pi$ is a lenient $\M$-inflation of $12$ or $21$ with $\tau_1$ and $\tau_2$, in some order.

Finally, if all of the non-empty cells of $\pi$ contain only one point, then label the point in any one non-empty cell of $\pi$ with the symbol $1$ and the point in any other non-empty cell with the symbol $2$. Now assign every other point in every other cell either the label $1$ or $2$ in such a way that, forming the gridded permutations $\tau_1$ and $\tau_2$ from the labels, $\pi$ can be expressed as a lenient inflation of $12$ or $21$ by $\tau_1$ and $\tau_2$ in some order.
\end{proof}

\section{Grid Classes by Symmetry}\label{sec-symmetry}

For the remainder of this paper we will be working towards proving Theorem~\ref{thm-not-pwo} by showing that certain types of grid class are not partially well-ordered. Among these non-partially well-ordered grid classes will be those needed to prove the remaining direction of Theorem~\ref{thm-fin-simple-grid}. We begin by showing how we may divide grid classes into families using ``grid mappings''.

Let $\M$ be an $m\times n$ gridding matrix, and let $\pi$ be an $\M$-gridded permutation. Recall that the \emph{inverse} of a permutation $\pi$ is $\pi^{-1}$, defined by $\pi^{-1}(i)=j$ if and only if $\pi(j)=i$, and we extend this in two ways: first to an $\M$-gridded permutation $\pi$ by mapping any vertical line between positions $i$ and $i+1$ ($i=0,\ldots,n$) to a horizontal line between values $i$ and $i+1$ and vice versa, and second to a permutation class $\C$ by setting $\C^{-1}=\{\pi^{-1}:\pi\in\C\}$. We consider the effect of taking the inverse of $\pi$ on the gridding of $\pi$, and consequently the effect on $\M$ of taking the inverse of $\grid(\M)$.

\begin{lemma}Let $\M$ be an $m\times n$ gridding matrix. Then $\grid(\M)^{-1} =\grid(\phi(\M))$ where $\phi(\M)$ is defined by $(\phi(\M))_{ij}=\M_{ji}^{-1}$.
\end{lemma}

We will call the map $\phi$ the \emph{grid inverse} map.

\begin{proof}
First note that $\phi(\phi(\M))=\M$, so it suffices to show that $\grid(\M)^{-1}\subseteq\grid(\phi(\M))$. Let $\pi$ be any permutation in $\grid(\M)^{-1}$, so $\pi^{-1}\in\grid(\M)$ is $\M$-griddable. Pick any $\M$-gridding of $\pi^{-1}$, and apply the inverse operation to this gridded matrix to recover a gridding of $\pi$. By definition, all points of the $ij$th cell of the gridded version of $\pi^{-1}$ are mapped under inverse to the $ji$th cell of the gridded $\pi$. Moreover, if the points in the $ij$th cell of $\pi^{-1}$ form the permutation $\sigma$, then it is clear that $\sigma^{-1}$ is the permutation formed by the points in the $ji$th cell of $\pi$, and so $\pi\in\grid(\phi(\M))$.
\end{proof}

Given a permutation $\pi$ of length $k$, the \emph{reverse} of $\pi$, written $r(\pi)$, is the permutation obtained by reading the entries of $\pi$ from right to left, i.e.\ for $i\in[k]$, we have $r(\pi)(i)=\pi(k+1-i)$. Similarly, the \emph{complement} of $\pi$, denoted $c(\pi)$, is formed by reading the permutation from top to bottom, i.e.\ $c(\pi)(i) = k+1-\pi(i)$. Accordingly, the \emph{reverse} of a set of permutations $X$ is $r(X)=\{r(\pi):\pi\in X\}$, and the \emph{complement} is $c(X)=\{c(\pi):\pi\in X\}$. Note in particular that if $\C=\av{B}$ is a permutation class with basis $B$ then $r(\C)=\av{r(B)}$ and $c(\C)=\av{c(B)}$.

Now let $\M$ be any $m\times n$ gridding matrix. For fixed $i\in[m]$, let $r_i(\M)$ be the \emph{$i$th column reverse} of $\M$, formed by applying the reverse map $r$ to every cell in column $i$. Thus for all $j\in[n]$, for any $i'\neq i$ we have $(r_i(\M))_{i'j}=\M_{i'j}$, while $(r_i(\M))_{ij}=r(\M_{ij})$. We define the \emph{$j$th row complement} analogously: $(c_j(\M))_{ij'}=\M_{ij'}$ whenever $j'\neq j$, and $(c_j(\M))_{ij}=c(\M_{ij})$ for all $i\in[m]$. Next, if $\mu$ is a permutation of length $m$, then let $\mu(\M)$ be the gridding matrix formed by permuting the columns of $\M$ as prescribed by $\mu$, so that $(\mu(\M))_{ij}=\M_{\mu(i)j}$. We say that $\mu$ is a \emph{permutation of the columns of $\M$}. Similarly, a \emph{permutation of the rows of $\M$} is a permutation $\nu$ of length $n$ satisfying $(\nu(\M))_{ij}=\M_{i\nu(j)}$.

We also extend the definitions of complements, reverses and permutations to gridded permutations in the obvious way. For example, if $\pi$ is a gridded permutation for which the set of points in row $j$ have values $a,a+1,\ldots,b$, then the $j$th row complement of $\pi$ is $c_j(\pi)$ defined by $c_j(\pi)(i)=b+a-\pi(i)$ if $(i,\pi(i))$ lies in row $j$, and $c_j(\pi)(i)=\pi(i)$ otherwise. See Figure~\ref{fig-row-complement}.

\begin{figure}
\begin{center}
\begin{tabular}{ccccc}
\psset{xunit=0.012in, yunit=0.012in}
\psset{linewidth=0.005in}
\begin{pspicture}(0,0)(120,120)
\psaxes[dy=10,Dy=1,dx=10,Dx=1,tickstyle=bottom,showorigin=false,labels=none](0,0)(120,120)
\pscircle*(10,50){0.04in}
\pscircle*(20,10){0.04in}
\pscircle*(30,40){0.04in}
\pscircle*(40,80){0.04in}
\pscircle*(50,90){0.04in}
\pscircle*(60,70){0.04in}
\pscircle*(70,110){0.04in}
\pscircle*(80,120){0.04in}
\pscircle*(90,20){0.04in}
\pscircle*(100,60){0.04in}
\pscircle*(110,30){0.04in}
\pscircle*(120,100){0.04in}
\psline[linestyle=dashed](0,45)(120,45)
\psline[linestyle=dashed](0,85)(120,85)
\psline[linestyle=dashed](45,0)(45,120)
\psline[linestyle=dashed](85,0)(85,120)
\end{pspicture}
&\rule{20pt}{0pt}&
\psset{xunit=0.012in, yunit=0.012in}
\psset{linewidth=0.005in}
\begin{pspicture}(0,0)(120,120)
\psaxes[dy=10,Dy=1,dx=10,Dx=1,tickstyle=bottom,showorigin=false,labels=none](0,0)(120,120)
\pscircle*(10,80){0.04in}
\pscircle*(20,10){0.04in}
\pscircle*(30,40){0.04in}
\pscircle*(40,50){0.04in}
\pscircle*(50,90){0.04in}
\pscircle*(60,60){0.04in}
\pscircle*(70,110){0.04in}
\pscircle*(80,120){0.04in}
\pscircle*(90,20){0.04in}
\pscircle*(100,70){0.04in}
\pscircle*(110,30){0.04in}
\pscircle*(120,100){0.04in}
\psline[linestyle=dashed](0,45)(120,45)
\psline[linestyle=dashed](0,85)(120,85)
\psline[linestyle=dashed](45,0)(45,120)
\psline[linestyle=dashed](85,0)(85,120)
\end{pspicture}
&\rule{20pt}{0pt}&
\psset{xunit=0.012in, yunit=0.012in}
\psset{linewidth=0.005in}
\begin{pspicture}(0,0)(120,120)
\psaxes[dy=10,Dy=1,dx=10,Dx=1,tickstyle=bottom,showorigin=false,labels=none](0,0)(120,120)
\pscircle*(10,20){0.04in}
\pscircle*(20,60){0.04in}
\pscircle*(30,30){0.04in}
\pscircle*(40,100){0.04in}
\pscircle*(50,50){0.04in}
\pscircle*(60,10){0.04in}
\pscircle*(70,40){0.04in}
\pscircle*(80,80){0.04in}
\pscircle*(90,90){0.04in}
\pscircle*(100,70){0.04in}
\pscircle*(110,110){0.04in}
\pscircle*(120,120){0.04in}
\psline[linestyle=dashed](0,45)(120,45)
\psline[linestyle=dashed](0,85)(120,85)
\psline[linestyle=dashed](45,0)(45,120)
\psline[linestyle=dashed](85,0)(85,120)
\end{pspicture}
\end{tabular}
\end{center}
\caption{From left to right, the $3\times 3$ gridded permutation $\pi=5\ 1\ 4\ 8\ 9\ 7\ 11\ 12\ 2\ 6\ 3\ 10$, the 2nd row complement $c_2(\pi)=\ 8\ 1\ 4\ 5\ 9\ 6\ 11\ 12\ 2\ 7\ 3\ 10$, and the permutation $\mu(\pi)=2\ 6\ 3\ 10\ 5\ 1\ 4\ 8\ 9\ 7\ 11\ 12$ where $\mu=312$ is a permutation of the columns.}\label{fig-row-complement}
\end{figure}
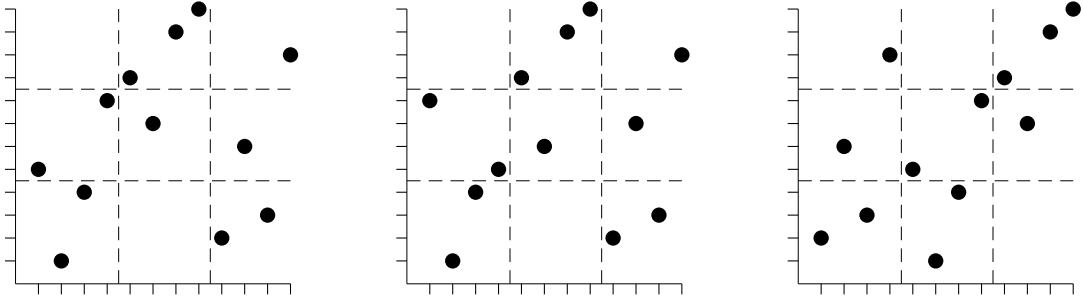

A \emph{grid mapping} is any composition of grid inverse, row complements, column reverses and row and column permutations, and we say that two matrices $\M$ and $\N$ are \emph{equivalent under the grid mapping $f$} if $f(\M)=\N$. (Note that this extends in a natural way to an equivalence relation.) Grid mappings do not in general preserve the normal permutation containment ordering, but they do respect gridded containment (defined in Section~\ref{sec-pwo-grid-classes}).

\begin{lemma}\label{lem-mapping-containment}Let $\M$ be a gridding matrix, $\alpha$ and $\pi$ two $\M$-gridded permutations and $f$ any grid mapping of $\M$. Then $\alpha\leq_\M\pi$ if and only if $f(\alpha)\leq_{f(\M)}f(\pi)$.\end{lemma}

\begin{proof}It suffices to show that $\alpha\leq_\M\pi$ implies $f(\alpha)\leq_{f(\M)}f(\pi)$ where $f$ is a grid inverse, row complement, column reverse, or a row or column permutation. We will consider only the grid inverse and column permutation cases, the others following by similar arguments.

Suppose that $\alpha$ is of length $k$ and $\pi$ of length $\ell$, and that the indices $1\leq i_1< \cdots< i_k\leq\ell$ give rise to a subsequence $\pi(i_1)\cdots\pi(i_k)$ that witnesses the gridded containment $\alpha\leq_\M\pi$. If $f=\phi$ is the grid inverse mapping then $\alpha\leq_\M\pi$ immediately implies $f(\alpha)=\alpha^{-1}\leq\pi^{-1}=f(\pi)$ as this is the normal inverse for permutations. Moreover, if $\pi(i_j)$ and $\alpha(j)$ ($j=1,\ldots,k$) lie in cell $st$ of $\M$, then their images under $f$ both lie in cell $ts$ of $f(\M)$, from which we conclude that $f(\alpha)\leq_{f(\M)}f(\pi)$.

Now suppose that $f$ is a column permutation, and note (by composing functions) that we can suppose that $f$ swaps two columns, $u$ and $v$, say. It is clear that the images of $\pi(i_j)$ and $\alpha(j)$ under $f$ both lie in the same cell, so it remains to show that $f(\alpha)\leq f(\pi)$ as ungridded permutations. This, however, is also straightforward: $f$ simply swaps the segments of $\alpha$ that lie in columns $u$ and $v$, and it does likewise in $\pi$. In particular, $f$ swaps the two subsequences of $\pi(i_1)\cdots\pi(i_k)$ lying in columns $u$ and $v$, and this image is a copy of $f(\alpha)$ in $f(\pi)$.
\end{proof}

We next make a simple observation, which allows us to pass between grid containment and normal permutation containment.

\begin{lemma}\label{lem-m-gridding-containment}Let $\alpha$ and $\pi$ be $\M$-griddable permutations with $\alpha\leq\pi$ as ungridded permutations. Then for any $\M$-gridding of $\pi$, there exists an $\M$-gridding of $\alpha$ such that $\alpha\leq_\M\pi$.\end{lemma}

\begin{proof}
This follows trivially by considering any subsequence of $\pi$ order isomorphic to $\alpha$, and then adding the $\M$-gridding to $\pi$, and hence to $\alpha$.
\end{proof}

We will use Lemma~\ref{lem-m-gridding-containment} on permutations that have a unique gridding: if $\alpha$ and $\pi$ are two permutations which have unique $\M$-griddings for some matrix $\M$, then $\alpha\not\leq_\M\pi$ implies $\alpha\not\leq\pi$. However, unique griddability is not in general preserved by grid mappings. For example, $135246$ has a unique gridding in $\grid(\av{21}\hspace{4pt}\av{21})$, but applying a column reverse to the first column yields the permutation $531246$, which can be gridded in two different ways in $\grid(\av{12}\hspace{4pt}\av{21})$. Thus, for a gridding matrix $\M$, we say that an $\M$-gridded permutation $\pi$ is \emph{strongly uniquely $\M$-griddable} if the given $\M$-gridding of $\pi$ is unique and, for every grid mapping $f$ of $\M$, $f(\pi)$ is also the unique $f(\M)$-gridding of $f(\pi)$. This extra condition gives us what we need:

\begin{theorem}\label{thm-strongly-uniquely-not-pwo}Let $\M$ be a gridding matrix, and let $A$ be an infinite antichain for which infinitely many elements are strongly uniquely $\M$-griddable. Then the grid class of any gridding matrix $\N$ that is equivalent to $\M$ under some grid mapping is not partially well-ordered.\end{theorem}

\begin{proof}First, we may assume that $A$ consists only of strongly uniquely $\M$-griddable permutations, as we may discard any elements that are not. Note that $\grid(\M)$ contains $A$ and so is not partially well-ordered. Let $f$ be any grid mapping of $\M$, and let $\N=f(\M)$.
Take any pair of distinct permutations $\alpha,\beta\in A$ (noting that $\alpha\not\leq\beta$), and equip each permutation with its unique $\M$-gridding. With these griddings $f(\alpha)$ and $f(\beta)$ are $\N$-gridded permutations, and since $\alpha$ and $\beta$ are strongly uniquely $\M$-griddable these $\N$-griddings are the unique griddings of the underlying permutations of $f(\alpha)$ and $f(\beta)$. Now, since $\alpha\not\leq\beta$ we have $\alpha\not\leq_\M\beta$, and consequently $f(\alpha)\not\leq_\N f(\beta)$ by Lemma~\ref{lem-mapping-containment}. Additionally, we have $f(\alpha)\not\leq f(\beta)$ as ungridded permutations by Lemma~\ref{lem-m-gridding-containment}. Similarly, $\beta\not\leq\alpha$ implies $f(\beta)\not\leq f(\alpha)$, and so $f(\alpha)$ and $f(\beta)$ are incomparable permutations lying in $\grid(\N)$, completing the proof.
\end{proof}

\section{A Family of Grid Matrices}\label{sec-family}

In this section we will reduce the number of gridding matrices that we need to consider to prove Theorem~\ref{thm-not-pwo} to a family in which we can easily build infinite antichains. In agreement with this Theorem, from now on we will now only consider gridding matrices where every non-empty cell is an infinite class, and for which the graph of the matrix either has a cycle, or a component with at least two non-monotone griddable classes. Let $\M$ be such a matrix. First note that if $G_\M$ contains a cycle then it contains as a subclass a cyclic monotone grid class, which is not partially well-ordered by Theorem~\ref{thm-murphy-vatter}. Thus we will assume from now on that $G_\M$ is acyclic and has a component with at least two non-monotone griddable entries.

Let $\F$ be the family of all gridding matrices $\M$ satisfying: (i) $G_\M$ is a path, (ii) the cells corresponding to the end points of $G_\M$ are either $\oplus 21$ or $\ominus 12$, and (iii) the cells corresponding to the internal vertices are monotone. First, we observe that it suffices to show that every grid class of a matrix from this family $\F$ contains an infinite antichain:

\begin{lemma}\label{lem-down-to-f}Every grid class $\grid(\M)$, for which $G_\M$ has a component with at least two non-monotone-griddable classes, contains $\grid(\M')$ for some $\M'\in\F$.\end{lemma}

\begin{proof}
Pick any two non-monotone-griddable cells lying in the same component of $G_\M$, and consider any path in $G_\M$ between these two cells. By Corollary~\ref{cor-non-monotone}, each of the two selected non-monotone-griddable cells contain $\oplus 21$ or $\ominus 12$, and all other cells on the path contain $\av{12}$ or $\av{21}$ (by Erd\H{o}s and Szekeres~\cite{es:acpig}, since we have assumed these cells are infinite). The Lemma now follows routinely by setting all cells of $\M$ not on this path to be empty (and deleting any resulting empty rows and columns), and replacing the cells on this path with appropriate subclasses.
\end{proof}

Now that we have the family of matrices $\F$, we consider the effect of grid mappings on this family. As observed earlier, grid mappings define an equivalence relation, and we wish to find a suitable representative from each class in which to build an antichain that satisfies Theorem~\ref{thm-strongly-uniquely-not-pwo}.

Let $\C=\D^+=\oplus 21$ and $\D^-=\ominus 12$. For $k\in\mathbb{N}$ define $\M^{k}$ recursively as follows:
\begin{itemize}
\item $\M^1=\left( \C\hspace{6pt} \D^-\right)$.
\item When $k=4\ell+1$, $\M^{k}$ is a $(2\ell+2)\times(2\ell+1)$ matrix with $\M^{k}_{ij}=\M^{k-1}_{ij}$ for all $i\in[1,2\ell+1],j\in[2,2\ell+1]$; $\M^{k}_{(2\ell+2)1}=\D^-$; $\M^{k}_{11}=\av{21}$; and all other entries are $\varnothing$.

\item When $k=4\ell+2$, $\M^{k}$ is a $(2\ell+2)\times(2\ell+2)$ matrix with $\M^{k}_{ij}=\M^{k-1}_{ij}$ for all $i,j\in[1,2\ell+1]$; $\M^{k}_{(2\ell+2)(2\ell+2)}=\D^+$; $\M^{k}_{(2\ell+2)1}=\av{12}$; and all other entries are $\varnothing$.

\item When $k=4\ell+3$, $\M^{k}$ is a $(2\ell+3)\times(2\ell+2)$ matrix with $\M^{k}_{ij}=\M^{k-1}_{(i-1)j}$ for all $i\in[2,2\ell+3],j\in[1,2\ell+1]$; $\M^{k}_{1(2\ell+2)}=\D^-$; $\M^{k}_{(2\ell+3)(2\ell+2)}=\av{21}$; and all other entries are $\varnothing$.

\item When $k=4\ell+4$, $\M^{k}$ is a $(2\ell+3)\times(2\ell+3)$ matrix with $\M^{k}_{ij}=\M^{k-1}_{i(j-1)}$ for all $i\in[1,2\ell+2],j\in[2,2\ell+3]$; $\M^{k}_{11}=\D^+$; $\M^{k}_{1(2\ell+3)}=\av{12}$; and all other entries are $\varnothing$.
\end{itemize}
Suppressing the labels of empty cells, the first few such matrices are:
\begin{align*}
\M^1&=\left(\C\hspace{6pt}\D^- \right)\\
\M^2&=\left(\begin{matrix}&\D^+\\\C&\av{12}\end{matrix}\right)\\
\M^3&=\left(\begin{matrix}\D^-&&\av{21}\\&\C&\av{12}\end{matrix}\right)\\
\M^4&=\left(\begin{matrix}\av{12}&&\av{21}\\&\C&\av{12}\\\D^+\end{matrix}\right)\\
\M^5&=\left(\begin{matrix}\av{12}&&\av{21}\\&\C&\av{12}\\\av{21}&&&\D^-\end{matrix}\right).
\end{align*}
Note that $G_{\M^k}$ is a path of length $k$, one end of which is labelled by $\C$ and the other by either $\D^-$ or $\D^+$, and whose internal vertices are labelled by $\av{21}$ or $\av{12}$.

\begin{theorem}\label{thm-matrix-symmetry}Every gridding matrix $\M\in \F$ is equivalent under some grid mapping to some $\M^k$.\end{theorem}

\begin{proof}
We will form the grid mapping $f:\M^k\rightarrow\M$ in three stages. First, we select $k$ and check whether we need to apply the grid inverse map to $\M^k$ so that it has the same dimensions as $\M$. Next, we permute the rows and columns of $\M^k$ (or $\phi(\M^k)$) to form an intermediate matrix $\N$ that has empty cells in exactly the same positions as $\M$. Finally, we use row complements and column reversals on $\N$ to match the non-empty cells to those of $\M$.

We pick $k$ so that $\M^k$ or its transpose has the same dimensions as $\M$. There are precisely two cases where we need to apply the grid inverse map $\phi$: The first is if $\M$ is not square and it has the same dimensions as the transpose of $\M$. The second case is where $\M$ is square, but both the first and last edges in the path $G_\M$ correspond to pairs of vertices which share a column: when $\M^k$ is square, both the first and last edges in the path $G_{\M^k}$ arise from pairs of vertices which share a row, hence the need for the grid inverse map.

Suppose without loss that we did not need to apply the grid inverse map. We now need to permute the rows and columns of $\M^k$ so that the non-empty cells are moved to the same places as those of $\M$. This, however, is straightforward: the graphs of the matrices $\M$ and $\M^k$ are the same, so we simply have to apply row and column transpositions to move each non-empty cell of $\M^k$ in turn to the required position. Call the resulting matrix $\N$.

All that remains is to fix the non-empty cells of $\N$ to have the same labels as $\M$. This can be done by starting at one end of the path, and applying successive row complements and/or column reverses so that each non-empty cell in turn is labelled correctly.
\end{proof}
\section{Grid Pin Sequences and Antichains}\label{sec-grid-antichains}

By the results of the previous section, to complete the proof of Theorem~\ref{thm-not-pwo}, all we require by Theorems~\ref{thm-strongly-uniquely-not-pwo} and~\ref{thm-matrix-symmetry} is to find a strongly uniquely $\M^k$-griddable antichain for each $k\in \mathbb{N}$.
The reason we chose the matrices $\M^k$ is that they admit antichains that are easily described in terms of ``grid pin sequences''. We now define these pin sequences, and prove some elementary results about them that should assist in our description of the antichains we wish to construct --- it is not our aim here to produce a complete theory of these sequences.

\begin{figure}
\begin{center}
\begin{tabular}{ccccc}
\psset{xunit=0.015in, yunit=0.015in}
\psset{linewidth=0.005in}
\begin{pspicture}(0,0)(110,110)
\psframe[linecolor=darkgray,fillstyle=solid,fillcolor=lightgray,linewidth=0.01in](25,95)(43,103)
\psline[linestyle=dashed](0,45)(110,45)
\psline[linestyle=dashed](0,95)(110,95)
\psline[linestyle=dashed](25,0)(25,110)
\psline[linestyle=dashed](65,0)(65,110)
\pscircle(25,95){0.04in}
\pscircle*(40,100){0.04in}
\pscircle*(30,80){0.04in}
\psline(30,80)(30,100)
\rput[l](45,100){$p_1$}
\rput[c](30,73){$p_2$}
\end{pspicture}
&\rule{3pt}{0pt}&
\psset{xunit=0.015in, yunit=0.015in}
\psset{linewidth=0.005in}
\begin{pspicture}(0,0)(110,110)
\psframe[linecolor=darkgray,fillstyle=solid,fillcolor=lightgray,linewidth=0.01in](27,77)(43,103)
\psline[linestyle=dashed](0,45)(110,45)
\psline[linestyle=dashed](0,95)(110,95)
\psline[linestyle=dashed](25,0)(25,110)
\psline[linestyle=dashed](65,0)(65,110)
\pscircle(25,95){0.04in}
\pscircle*(40,100){0.04in}
\pscircle*(30,80){0.04in}
\pscircle*(80,90){0.04in}
\psline(30,80)(30,100)
\psline(27,90)(80,90)
\rput[l](45,100){$p_1$}
\rput[c](30,73){$p_2$}
\rput[l](85,90){$p_3$}
\end{pspicture}
&\rule{3pt}{0pt}&
\psset{xunit=0.015in, yunit=0.015in}
\psset{linewidth=0.005in}
\begin{pspicture}(0,0)(110,110)
\psframe[linecolor=darkgray,fillstyle=solid,fillcolor=lightgray,linewidth=0.01in](27,77)(83,93)
\psline[linestyle=dashed](0,45)(110,45)
\psline[linestyle=dashed](0,95)(110,95)
\psline[linestyle=dashed](25,0)(25,110)
\psline[linestyle=dashed](65,0)(65,110)
\pscircle(25,95){0.04in}
\pscircle*(40,100){0.04in}
\pscircle*(30,80){0.04in}
\pscircle*(80,90){0.04in}
\pscircle*(70,60){0.04in}
\psline(30,80)(30,100)
\psline(27,90)(80,90)
\psline(70,95)(70,60)
\rput[c](30,73){$p_2$}
\rput[l](85,90){$p_3$}
\rput[l](70,53){$p_4$}
\end{pspicture}
\\\\
\psset{xunit=0.015in, yunit=0.015in}
\psset{linewidth=0.005in}
\begin{pspicture}(0,0)(110,110)
\psframe[linecolor=darkgray,fillstyle=solid,fillcolor=lightgray,linewidth=0.01in](67,57)(83,93)
\psline[linestyle=dashed](0,45)(110,45)
\psline[linestyle=dashed](0,95)(110,95)
\psline[linestyle=dashed](25,0)(25,110)
\psline[linestyle=dashed](65,0)(65,110)
\pscircle(25,95){0.04in}
\pscircle*(40,100){0.04in}
\pscircle*(30,80){0.04in}
\pscircle*(80,90){0.04in}
\pscircle*(70,60){0.04in}
\pscircle*(10,70){0.04in}
\psline(30,80)(30,100)
\psline(27,90)(80,90)
\psline(70,95)(70,60)
\psline(10,70)(75,70)
\rput[l](85,90){$p_3$}
\rput[c](70,53){$p_4$}
\rput[r](5,70){$p_5$}
\end{pspicture}
&&
\psset{xunit=0.015in, yunit=0.015in}
\psset{linewidth=0.005in}
\begin{pspicture}(0,0)(110,110)
\psframe[linecolor=darkgray,fillstyle=solid,fillcolor=lightgray,linewidth=0.01in](7,57)(73,73)
\psline[linestyle=dashed](0,45)(110,45)
\psline[linestyle=dashed](0,95)(110,95)
\psline[linestyle=dashed](25,0)(25,110)
\psline[linestyle=dashed](65,0)(65,110)
\pscircle(25,95){0.04in}
\pscircle*(40,100){0.04in}
\pscircle*(30,80){0.04in}
\pscircle*(80,90){0.04in}
\pscircle*(70,60){0.04in}
\pscircle*(10,70){0.04in}
\pscircle*(20,30){0.04in}
\psline(30,80)(30,100)
\psline(27,90)(80,90)
\psline(70,95)(70,60)
\psline(10,70)(75,70)
\psline(20,30)(20,75)
\rput[c](70,53){$p_4$}
\rput[r](5,70){$p_5$}
\rput[c](20,23){$p_6$}
\end{pspicture}
&&
\psset{xunit=0.015in, yunit=0.015in}
\psset{linewidth=0.005in}
\begin{pspicture}(0,0)(110,110)
\psframe[linecolor=darkgray,fillstyle=solid,fillcolor=lightgray,linewidth=0.01in](7,27)(23,73)
\psline[linestyle=dashed](0,45)(110,45)
\psline[linestyle=dashed](0,95)(110,95)
\psline[linestyle=dashed](25,0)(25,110)
\psline[linestyle=dashed](65,0)(65,110)
\pscircle(25,95){0.04in}
\pscircle*(40,100){0.04in}
\pscircle*(30,80){0.04in}
\pscircle*(80,90){0.04in}
\pscircle*(70,60){0.04in}
\pscircle*(10,70){0.04in}
\pscircle*(20,30){0.04in}
\pscircle*(60,40){0.04in}
\psline(30,80)(30,100)
\psline(27,90)(80,90)
\psline(70,95)(70,60)
\psline(10,70)(75,70)
\psline(20,30)(20,75)
\psline(15,40)(60,40)
\rput[r](5,70){$p_5$}
\rput[c](20,23){$p_6$}
\rput[l](65,40){$p_7$}
\end{pspicture}
\\\\
\psset{xunit=0.015in, yunit=0.015in}
\psset{linewidth=0.005in}
\begin{pspicture}(0,0)(110,110)
\psframe[linecolor=darkgray,fillstyle=solid,fillcolor=lightgray,linewidth=0.01in](17,27)(63,43)
\psline[linestyle=dashed](0,45)(110,45)
\psline[linestyle=dashed](0,95)(110,95)
\psline[linestyle=dashed](25,0)(25,110)
\psline[linestyle=dashed](65,0)(65,110)
\pscircle(25,95){0.04in}
\pscircle*(40,100){0.04in}
\pscircle*(30,80){0.04in}
\pscircle*(80,90){0.04in}
\pscircle*(70,60){0.04in}
\pscircle*(10,70){0.04in}
\pscircle*(20,30){0.04in}
\pscircle*(60,40){0.04in}
\pscircle*(50,10){0.04in}
\psline(30,80)(30,100)
\psline(27,90)(80,90)
\psline(70,95)(70,60)
\psline(10,70)(75,70)
\psline(20,30)(20,75)
\psline(15,40)(60,40)
\psline(50,10)(50,45)
\rput[c](20,23){$p_6$}
\rput[l](65,40){$p_7$}
\rput[c](50,3){$p_8$}
\end{pspicture}
&&
\psset{xunit=0.015in, yunit=0.015in}
\psset{linewidth=0.005in}
\begin{pspicture}(0,0)(110,110)
\psframe[linecolor=darkgray,fillstyle=solid,fillcolor=lightgray,linewidth=0.01in](47,7)(63,43)
\psline[linestyle=dashed](0,45)(110,45)
\psline[linestyle=dashed](0,95)(110,95)
\psline[linestyle=dashed](25,0)(25,110)
\psline[linestyle=dashed](65,0)(65,110)
\pscircle(25,95){0.04in}
\pscircle*(40,100){0.04in}
\pscircle*(30,80){0.04in}
\pscircle*(80,90){0.04in}
\pscircle*(70,60){0.04in}
\pscircle*(10,70){0.04in}
\pscircle*(20,30){0.04in}
\pscircle*(60,40){0.04in}
\pscircle*(50,10){0.04in}
\pscircle*(100,20){0.04in}
\psline(30,80)(30,100)
\psline(27,90)(80,90)
\psline(70,95)(70,60)
\psline(10,70)(75,70)
\psline(20,30)(20,75)
\psline(15,40)(60,40)
\psline(50,10)(50,45)
\psline(45,20)(100,20)
\rput[l](65,40){$p_7$}
\rput[c](50,3){$p_8$}
\rput[l](105,20){$p_9$}
\end{pspicture}
&&
\psset{xunit=0.015in, yunit=0.015in}
\psset{linewidth=0.005in}
\begin{pspicture}(0,0)(110,110)
\psframe[linecolor=darkgray,fillstyle=solid,fillcolor=lightgray,linewidth=0.01in](47,7)(103,23)
\psline[linestyle=dashed](0,45)(110,45)
\psline[linestyle=dashed](0,95)(110,95)
\psline[linestyle=dashed](25,0)(25,110)
\psline[linestyle=dashed](65,0)(65,110)
\pscircle(25,95){0.04in}
\pscircle*(40,100){0.04in}
\pscircle*(30,80){0.04in}
\pscircle*(80,90){0.04in}
\pscircle*(70,60){0.04in}
\pscircle*(10,70){0.04in}
\pscircle*(20,30){0.04in}
\pscircle*(60,40){0.04in}
\pscircle*(50,10){0.04in}
\pscircle*(100,20){0.04in}
\pscircle*(90,50){0.04in}
\psline(30,80)(30,100)
\psline(27,90)(80,90)
\psline(70,95)(70,60)
\psline(10,70)(75,70)
\psline(20,30)(20,75)
\psline(15,40)(60,40)
\psline(50,10)(50,45)
\psline(45,20)(100,20)
\psline(90,15)(90,50)
\rput[c](50,3){$p_8$}
\rput[l](105,20){$p_9$}
\rput[l](93,57){$p_{10}$}
\end{pspicture}
\end{tabular}
\end{center}
\caption{A grid pin sequence on the $3\times 3$ grid.}
\label{fig-grid-pins}
\end{figure}
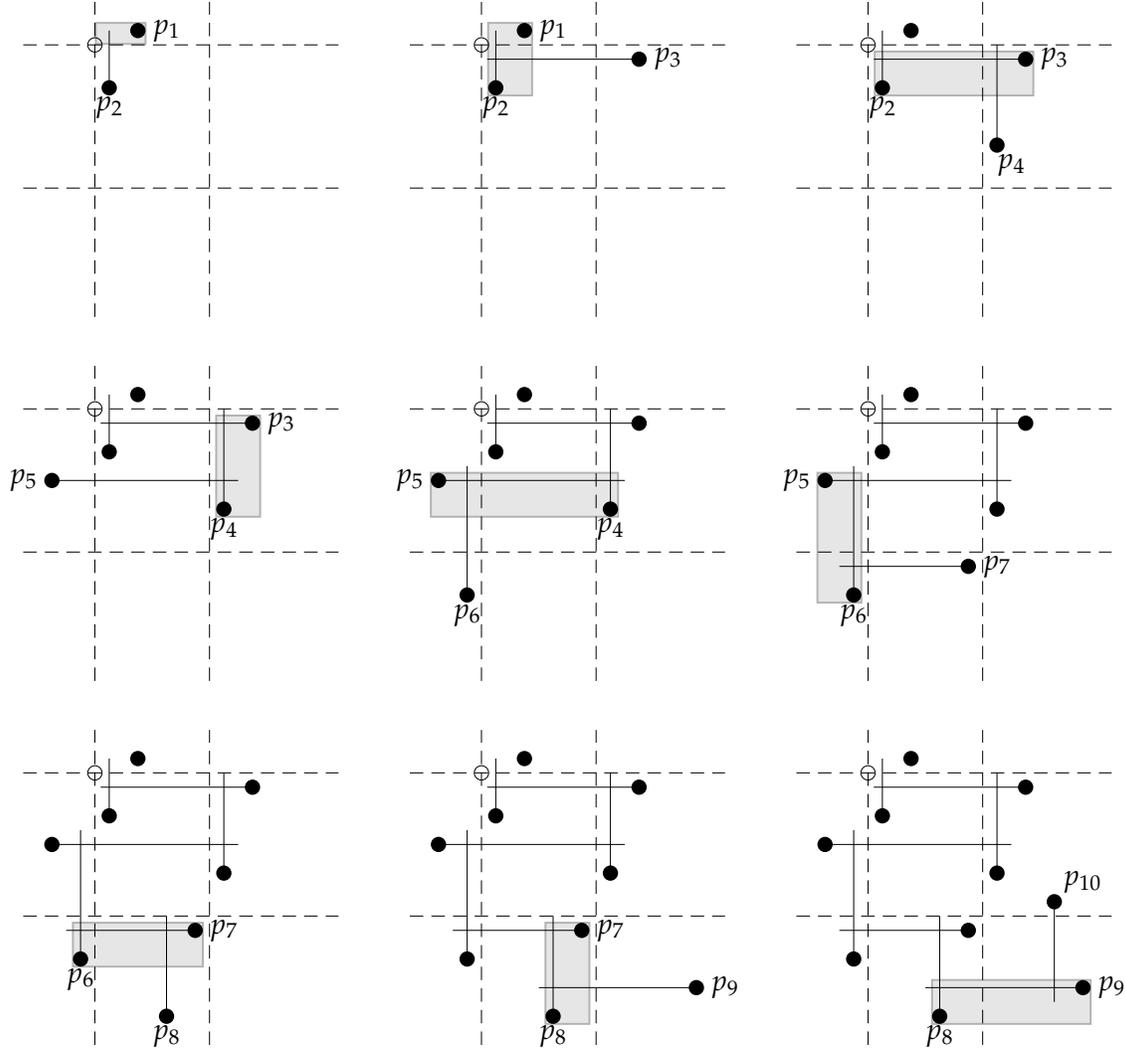
Given points $p_1,p_2,\ldots$ in the plane, denote by $\rect(p_1,p_2,\ldots)$ the smallest axes-parallel rectangle containing them. A \emph{grid pin sequence} is a sequence of points (called \emph{pins}) $p_1,p_2,\ldots$ in an $m\times n$ gridded plane which for $i\geq 2$ must satisfy four conditions:

\begin{itemize}
\item \emph{Local separation}: Each pin $p_{i+1}$ separates $p_i$ from $p_{i-1}$ by position or by value.
\item \emph{Local externality}: Each pin $p_{i+1}$ lies outside all of $\rect(p_0,p_1)$, $\rect(p_{1},p_{2})$, $\ldots$, $\rect(p_{i-1},p_i)$. The \emph{direction} of $p_{i+1}$ denotes its placement relative to $\rect(p_{i-1},p_i)$: if $p_{i+1}$ lies above (respectively, below, to the left, or to the right) $\rect(p_{i-1},p_i)$, then $p_{i+1}$ is an up (respectively, down, left, right) pin.
\item \emph{Row-column agreement}: If $p_{i+1}$ is an up or a down pin, it must lie in the same column as $p_{i}$, while if $p_{i+1}$ is a left or a right pin, it must lie in the same row.
\item \emph{Non-interaction}: Each pin $p_{i+1}$, could not have been used as a grid pin earlier in the pin sequence. I.e.\ for every $2\leq j<i$ the pin $p_{i+1}$ must violate one of local separation or row-column agreement with respect to $p_{j}$ and $p_{j-1}$ (note that it cannot violate local externality).
\end{itemize}

It still remains to explain how to initiate a grid pin sequence. We begin by placing a fictional pin $p_0$ corresponding to an \emph{origin} at the intersection of two chosen perpendicular grid lines. Our next pin, $p_1$, is then placed in one of the four cells adjacent to this origin and has two directions given by its position relative to $p_0$. For example, if $p_1$ lies below and to the left of $p_0$, then $p_1$ is both a left pin and a down pin. The second pin is then placed to satisfy the four above conditions relative to $p_0$ and $p_1$.

Grid pin sequences should be thought of as a generalisation of proper pin sequences, which were introduced in~\cite{brignall:simple-permutat:a} in a Ramsey-type argument on simple permutations. Proper pin sequences can be recovered from our definition by restricting our view to a $2\times 2$ grid: local separation and row-column agreement combine to form the separation condition of~\cite{brignall:simple-permutat:a}, and local externality and non-interaction combine to give the externality condition.

The directions of pins $p_2,p_3,\ldots$ must alternate:

\begin{lemma}\label{lem-pin-direction}In a grid pin sequence, if $p_i$ ($i>1$) is a left or a right pin, then $p_{i+1}$ is an up or a down pin. Similarly, if $p_i$ is an up or a down pin then $p_{i+1}$ must be a left or a right pin.\end{lemma}

\begin{proof}
Suppose without loss that $p_i$ is a left pin. By local externality and local separation, $p_{i+1}$ must extend from $\rect(p_{i-1},p_i)$.  However, if $p_{i+1}$ is a left or a right pin, then $p_{i+1}$ also either extends from $\rect(p_{i-2},p_{i-1})$ contradicting non-interaction, or it lies in $\rect(p_{i-2},p_{i-1})$ contradicting local externality.
\end{proof}

\begin{lemma}\label{lem-pin-placement}If $p_{i+1}$ is a left pin for the grid pin sequence $p_1,\ldots,p_i$, then $p_{i+1}$ lies further left than all previous left pins in its column, and to the right of all previous right pins in its column. Analogous statements hold if $p_{i+1}$ is a right, up or down pin.\end{lemma}

\begin{proof}
We prove both statements of the first sentence simultaneously by induction on the total number of left and right pins in a given column. The base case, where there is just a single left or right pin, is trivial. So now suppose for a contradiction to the first statement that $p_{i+1}$ is a left pin which lies to the right of some earlier left pin $p_{j}$ ($j<i$) in the same column, and assume without loss that there are no other left pins between $p_{i+1}$ and $p_{j}$. If $p_{j-1}$ (the predecessor of $p_j$) lies to the right of $p_{i+1}$, then $p_{i+1}$ separates $p_{j-1}$ from $p_j$, is not contained in any of $\rect(p_{j-1},p_{j})$,$\ldots$,$\rect(p_{0},p_{1})$ and shares a column with $p_j$, and hence is a pin for $p_1,\ldots,p_j$, contradicting non-interaction. Thus $p_{j-1}$ lies between $p_j$ and $p_{i+1}$ and so lies in the same column. Now consider the pin $p_{j-2}$, which was a left or a right pin, or $p_0$. (Note that $p_1$ is always either a left or a right pin.) It cannot be a left pin as it would necessarily have to lie by position between $p_j$ and $p_{i+1}$ but by our assumption there are no left pins between $p_j$ and $p_{i+1}$; it cannot be a right pin since by row-column agreement it must lie in the same column as $p_{j-1}$ but to the right of $p_j$, contradicting the inductive hypothesis; finally, it cannot be $p_0$ as then, in order to lie on an adjacent grid line and ensure that $p_j=p_2$ extends from $\rect(p_0,p_1)$, it must lie to the right of $p_{i+1}$, but then $p_{i+1}$ either lies in $\rect(p_0,p_1)$ contradicting local externality or it satisfies the conditions to be a pin for $p_0,p_1$ contradicting non-interaction.

A similar argument may be applied to show that $p_{i+1}$ lies to the right of all previous right pins in its column, and so by induction the first sentence of the lemma is true. Finally, symmetry proves the analogous statements in the other three directions.
\end{proof}

Unlike the $2\times 2$ case, for an arbitrary $m\times n$ grid the direction of the pin is not sufficient to describe the placement of the pin, so we need to be more specific. A \emph{horizontal pin} is either a left or a right pin, while a \emph{vertical pin} is either an up or a down pin.

\begin{lemma}\label{lem-pin-row-column}Let $p_1,p_2,\ldots, p_i$ be a grid pin sequence of length $i\geq 2$ in an $m\times n$ grid. Then if $p_{i+1}$ is a horizontal pin, its placement relative to $p_1,\ldots,p_i$ is uniquely determined (up to order isomorphism) by the column in which it lies. Similarly, if $p_{i+1}$ is a vertical pin, its placement relative to $p_1,\ldots,p_i$ is uniquely determined by the row in which it lies.\end{lemma}

\begin{figure}
\begin{center}
\psset{xunit=0.012in, yunit=0.012in}
\psset{linewidth=0.005in}
\begin{pspicture}(0,0)(190,170)
{\psset{fillstyle=solid,cornersize=absolute,linearc=1pt,linewidth=0.005in,linestyle=solid,linecolor=black}
  \psframe[fillcolor=lightgray](60,80)(100,100)}
\pscircle*(0,70){0.03in}
\pscircle*(20,80){0.03in}
\pscircle*(50,160){0.03in}
\pscircle*(60,140){0.03in}
\pscircle*(100,40){0.03in}
\pscircle*(110,10){0.03in}
\pscircle*(140,50){0.03in}
\pscircle*(150,100){0.03in}
{\tiny
\rput[c](155,95){$p_{i}$}
\rput[c](140,45){$p_{i-1}$}}
\psline[linestyle=dashed](40,0)(40,160)
\psline[linestyle=dashed](120,0)(120,160)
\psline[linestyle=dashed](0,60)(180,60)
\psline[linestyle=dashed](0,120)(180,120)
\psline[linestyle=dotted,dotsep=2pt](60,0)(60,160)
\psline[linestyle=dotted,dotsep=2pt](100,0)(100,160)
\psline[linestyle=dotted,dotsep=2pt](0,80)(180,80)
\psline[linestyle=dotted,dotsep=2pt](0,100)(180,100)
\psline(20,50)(20,80)
\psline(30,70)(0,70)
\psline(10,80)(10,50)
\psline(30,140)(60,140)
\psline(50,130)(50,160)
\psline(60,150)(30,150)
\psline(100,40)(130,40)
\psline(110,10)(110,50)
\psline(100,20)(130,20)
\psline(170,50)(140,50)
\psline(150,40)(150,100)
\psline[linestyle=dotted,dotsep=1pt](10,50)(10,40)
\psline[linestyle=dotted,dotsep=1pt](20,50)(20,40)
\psline[linestyle=dotted,dotsep=1pt](30,140)(20,140)
\psline[linestyle=dotted,dotsep=1pt](30,150)(20,150)
\psline[linestyle=dotted,dotsep=1pt](130,40)(140,40)
\psline[linestyle=dotted,dotsep=1pt](130,20)(140,20)
\psline[linestyle=dotted,dotsep=1pt](170,50)(180,50)
\end{pspicture}
\end{center}
\caption{The shaded region denotes the area where pin $p_{i+1}$ can be placed.}
\label{fig-pin-placement}
\end{figure}
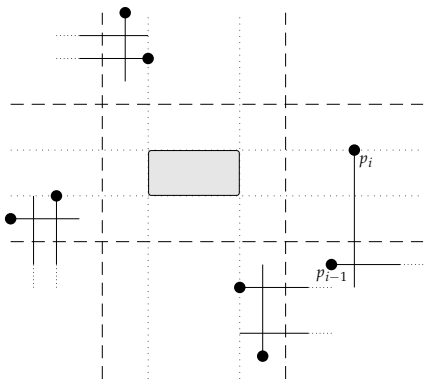

\begin{proof}
We prove only the case where $p_{i+1}$ is a horizontal pin and $p_i$ is an up pin. By row-column agreement, $p_{i+1}$ must be made to lie in the same row as $p_{i}$, so coupling this information with the knowledge that $p_{i+1}$ must lie in a specified column is enough to determine the cell into which $p_i$ is placed. In particular, if the column that is to contain $p_{i+1}$ is to the left (respectively, right) of the column containing $p_i$, then $p_{i+1}$ is a left (resp. right) pin. If $p_{i+1}$ is to lie in the same column as $p_i$, then the direction of $p_{i+1}$ must match the direction of $p_{i-1}$ to satisfy Lemma~\ref{lem-pin-row-column}. (Note that if $p_{i-1}=p_1$, then the direction of $p_{i+1}$ matches the horizontal direction of $p_1$.)

By Lemma~\ref{lem-pin-placement}, $p_i$ must be placed in the region of the cell to the left of all earlier left pins in its column, and to the right of all right pins in the column. This defines a vertical strip extending the length of the column that is devoid of pins. Similarly, $p_{i+1}$ must lie below $p_i$ and above $p_{i-1}$ to satisfy separation, and additionally it must lie above all up pins other than $p_i$ in its row to satisfy non-interaction. This defines a horizontal strip extending to the ends of the row which is devoid of pins.

The intersection of the horizontal strip and the vertical strip defines a rectangular region in the correct cell in which $p_{i+1}$ can be placed --- see Figure~\ref{fig-pin-placement}. By its construction, there are no points among $p_1,\ldots,p_i$ separating this region, and so all placements of $p_{i+1}$ within this region produce the same permutation up to order isomorphism.
\end{proof}

Note that the above lemma can be extended to include pin $p_2$, but this requires a little further thought. It is not sufficient to state which cell it is to be placed in as there are two different placements of $p_2$ if it is to lie in the same cell as $p_1$: one horizontal, one vertical. However, if the placement of $p_2$ is specified by a row, then we know $p_2$ is to be a vertical pin lying in the same column as $p_1$, and if specified by a column then $p_2$ is a horizontal pin lying in the same row as $p_1$.

Before we embark on constructing our antichain, we recall the definition of an inflation from Section~\ref{sec-definitions} and extend this to grid pin sequences. (Note that this extension of the definition of an inflation is unrelated to the one used in Section~\ref{sec-pwo-grid-classes}.) Letting $p_1,\ldots,p_n$ be a grid pin sequence, the \emph{grid pin sequence inflation of $p_1,\ldots,p_n$ by the permutations $\alpha_1,\ldots,\alpha_n$} is the permutation formed by taking the permutation corresponding to the grid pin sequence $p_1,\ldots,p_n$, and inflating each point $p_i$ ($i=1,\ldots, n$) with the permutation $\alpha_i$. This is denoted $p_1[\alpha_1],p_2[\alpha_2],\ldots,p_i[\alpha_i]$, but whenever $\alpha_i=1$ we denote the trivially inflated pin $p_i[1]$ simply by $p_i$. We call such a permutation an \emph{inflated grid pin permutation}.

For each $k$, we now use inflated grid pin sequences to construct an infinite set of permutations $A^k$ lying in $\grid(\M^k)$. This construction is accompanied by Figure~\ref{fig-grid-antichain}. We begin by showing how to construct the infinite uninflated grid pin sequence $p_1,p_2,\ldots$ that will be used to construct all the permutations of $A^k$: First place the imaginary pin $p_0$ in the top-right corner of the cell labelled $\C$ lying in the middle of $\M^k$, and the pin $p_1$ as a left and down pin (also in the cell labelled by $\C$). This cell is the only one in its column, but there is one other non-empty cell in the same row, into which we place a right pin $p_2$. We then recursively place each pin $p_{i+1}$ so that it does not lie in the same cell as $p_i$, but shares a row or column with $p_i$. (Note that by Lemma~\ref{lem-pin-row-column}, this is a sufficient description, as we know whether $p_i$ was a horizontal or a vertical pin.) Once we have placed our first pin $p_j$ in the cell labelled by $\D^+$ or $\D^-$, we place the next pin $p_{j+1}$ in the same cell, and then $p_{j+2}$ is placed in the cell that contained $p_{j-1}$. Again we place one pin per cell back around until we reach the cell labelled by $\C$. Once in the cell with label $\C$, we place a second point in this cell to ``turn around'', and repeat.

Finally, $A^k=\{\alpha_1,\alpha_2,\ldots\}$, where $\alpha_i=p_1[21],p_2, \ldots, p_{(2i-1)k-1}, p_{(2i-1)k}[\beta]$ is an inflated grid pin permutation of length $(2i-1)k+2$, with $\beta=21$ if $k$ is even and $\beta=12$ otherwise.

\begin{figure}
\begin{center}
\psset{xunit=0.009in, yunit=0.009in}
\psset{linewidth=0.005in}
\begin{pspicture}(0,0)(300,300)
\psaxes[dy=10,dx=10,tickstyle=bottom,showorigin=false,labels=none](0,0)(300,300)
\pscircle*(10,280){0.03in}\pscircle*(20,20){0.03in}
\pscircle*(30,10){0.03in}\pscircle*(40,40){0.03in}
\pscircle*(50,270){0.03in}\pscircle*(60,240){0.03in}
\pscircle*(70,30){0.03in}\pscircle*(80,220){0.03in}
\pscircle*(90,50){0.03in}\pscircle*(100,80){0.03in}
\pscircle*(110,210){0.03in}\pscircle*(120,180){0.03in}
\pscircle*(130,90){0.03in}\pscircle*(140,140){0.03in}
\pscircle*(150,110){0.03in}\pscircle*(160,160){0.03in}
\pscircle*(170,150){0.03in}\pscircle*(180,190){0.03in}
\pscircle*(190,170){0.03in}\pscircle*(200,130){0.03in}
\pscircle*(210,200){0.03in}\pscircle*(220,230){0.03in}
\pscircle*(230,120){0.03in}\pscircle*(240,250){0.03in}
\pscircle*(250,100){0.03in}\pscircle*(260,70){0.03in}
\pscircle*(270,260){0.03in}\pscircle*(280,290){0.03in}
\pscircle*(290,60){0.03in}
{\tiny
\rput[c](10,270){$p_{26}$}\rput[c](10,30){$p_{27}$}
\rput[c](40,30){$p_{10}$}\rput[c](40,270){$p_{11}$}
\rput[c](60,250){$p_{8}$}\rput[c](70,20){$p_{9}$}
\rput[c](81,210){$p_{22}$}\rput[c](101,50){$p_{23}$}
\rput[c](100,70){$p_{14}$}\rput[c](100,210){$p_{15}$}
\rput[c](120,190){$p_{4}$}\rput[c](120,90){$p_{5}$}
\rput[c](141,150){$p_{18}$}\rput[c](161,110){$p_{19}$}
\rput[c](149,160){$p_{1}$}\rput[c](190,190){$p_{3}$}
\rput[c](190,160){$p_{2}$}\rput[c](210,130){$p_{17}$}
\rput[c](210,210){$p_{16}$}\rput[c](210,230){$p_{21}$}
\rput[c](229,130){$p_{20}$}\rput[c](250,250){$p_{7}$}
\rput[c](250,90){$p_{6}$}\rput[c](270,70){$p_{13}$}
\rput[c](270,270){$p_{12}$}\rput[c](270,290){$p_{25}$}
\rput[c](290,70){$p_{24}$}}
\psline[linestyle=dashed](0,45)(300,45)
\psline[linestyle=dashed](0,105)(300,105)
\psline[linestyle=dashed](0,175)(300,175)
\psline[linestyle=dashed](0,235)(300,235)
\psline[linestyle=dashed](75,0)(75,300)
\psline[linestyle=dashed](135,0)(135,300)
\psline[linestyle=dashed](175,0)(175,300)
\psline[linestyle=dashed](235,0)(235,300)
\psline(150,170)(190,170)\psline(180,160)(180,190)
\psline(190,180)(120,180)\psline(130,190)(130,90)
\psline(120,100)(250,100)\psline(240,90)(240,250)
\psline(250,240)(60,240)\psline(70,250)(70,30)
\psline(80,40)(40,40)\psline(50,30)(50,270)
\psline(40,260)(270,260)\psline(260,270)(260,70)
\psline(270,80)(100,80)\psline(110,70)(110,210)
\psline(100,200)(210,200)\psline(200,210)(200,130)
\psline(210,140)(140,140)\psline(150,150)(150,110)
\psline(140,120)(230,120)\psline(220,110)(220,230)
\psline(230,220)(80,220)\psline(90,230)(90,50)
\psline(80,60)(290,60)\psline(280,50)(280,290)
\psline(290,280)(10,280)\psline(20,290)(20,26)
\psccurve(15,25)(20,10)(35,5)(30,20)
\psccurve(155,165)(160,150)(175,145)(170,160)
\end{pspicture}
\end{center}
\caption{An element of $A^8$ in the grid class of $\M^8$.}
\label{fig-grid-antichain}
\end{figure}
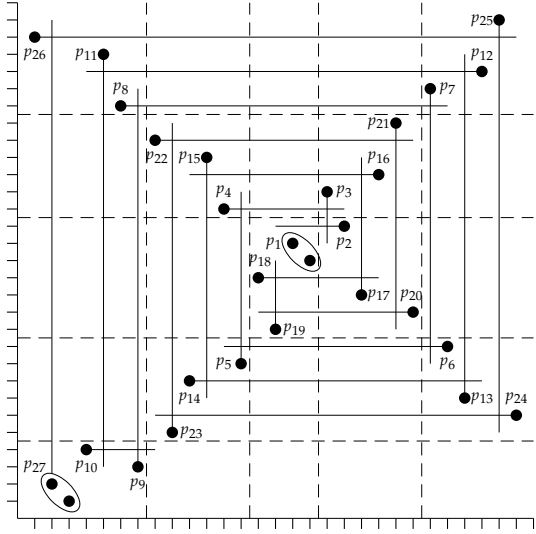

\begin{lemma}Every permutation in $A^k$ is $\M^k$-griddable.\end{lemma}

\begin{proof}
This is clear by considering the permitted region in which to place successive pins described in Lemma~\ref{lem-pin-row-column}. In particular, cells labelled by monotone classes contain only monotone sequences of the right type, and the non-monotone cells contain permutations from $\oplus 21$ or $\ominus 12$ as required.
\end{proof}

We need an infinite subset of $A^k$ that is both an antichain and strongly uniquely $\M^k$-griddable. We will in fact find the latter first: knowing the uniqueness of the $\M^k$-gridding will assist us in proving that elements of $A^k$ are incomparable. The methods of our proofs are similar in flavour to those used by Murphy and Vatter~\cite{murphy:profile-classes:}. We begin by making the following straightforward observation, which we will use repeatedly.

\begin{lemma}\label{lem-pin-pairs}
Let $\alpha\in A^k$ be of length $n+2$, $\alpha'$ be the $\M^k$-gridded permutation of length $n$ corresponding to the uninflated grid pin sequence $p_1,\ldots,p_n$ of $\alpha$, and $f$ be any grid mapping of $\M^k$. Then, if in $f(\alpha')$ the pins $p_i$ and $p_j$ ($1<i<j<n$) lie in the same cell and are adjacent by position, they are separated precisely by the pins $p_{i-1}$ and $p_{j+1}$ by value. The same holds swapping ``position'' and ``value''.
\end{lemma}

\begin{proof}
First, by composition of functions it is sufficient to prove the statement when $f$ is the grid inverse map, a row or column permutation, a row complement, a column reversal or the identity map. The effect of the grid inverse mapping $\phi$ of $\M^k$ on $\alpha'$ is merely to swap the terms ``position'' and ``value'' in the statement of the lemma, so we can discount this case. Moreover, every other grid mapping that we need to consider preserves the relative orderings of points by position and value in any given row or column, except possibly to reverse the order. Thus the lemma is true if we can show it is true when $f$ is the identity grid mapping, and this is easily seen by considering the placement of successive pins as described in Lemma~\ref{lem-pin-row-column}.
\end{proof}

This lemma is all that is required to prove what we need.

\begin{lemma}\label{lem-antichain-uniquely-griddable}Every permutation of length at least $2(k+1)^2+3$ in $A^k$ is strongly uniquely $\M^k$-griddable.\end{lemma}

\begin{proof}
Given $\alpha\in A^k$ of length $n+2\geq 2(k+1)^2+3$, consider the permutation $\alpha'$ corresponding to the uninflated grid pin sequence of length $n\geq 2(k+1)^2+1$ used to create $\alpha$. We will prove the result for $\alpha'$, from which the required result easily follows. Label the points of $f(\alpha')$ with $p_1,p_2,\ldots, p_n$ according to the grid pin sequence used to construct $\alpha'$, i.e.\ the point with label $p_i$ in $f(\alpha')$ is the image under $f$ of the pin $p_i$. Our approach is to find pairs of points that are adjacent by position or value and lie together in a single cell in both griddings, then apply Lemma~\ref{lem-pin-pairs} to find another pair with this property, and repeat the process. Note that in what follows, it does not particularly matter precisely what the map $f$ is: we will simply be exploiting particular features of $f(\M^k)$, for example $G_{f(\M_k)}$ is a path of length $k$ whose internal vertices are labelled by monotone classes and whose end vertices are $\oplus 21$ or $\ominus 12$.

Suppose that $f(\alpha')$ has two $f(\M^k)$-griddings, the first being the gridding inherited from the original gridding of $\alpha'$ under the grid mapping $f$, and the second some other gridding. Since $f(\alpha')$ contains at least $2(k+1)^2+1$ points, the second gridding has some cell that contains at least $2(k+1)+1$ points. Of these points, at least three must lie together in some cell in the first gridding, and so we can identify three --- $p_h$, $p_i$ and $p_j$ with $h<i<j$ --- which in each gridding lie in a common cell. (Note that these two common cells do not yet need to correspond to the same cell of $f(\M^k)$.) In each gridding, since $\rect(p_h,p_i,p_j)$ is necessarily contained within the cell, so any other point inside $\rect(p_h,p_i,p_j)$ must lie in the cell. Thus, by shrinking the rectangle and relabelling if necessary, we can assume that $\rect(p_h,p_i,p_j)$ contains only $p_h$, $p_i$ and $p_j$. Our first aim is to show that the griddings ``line up'': i.e.\ $p_h$, $p_i$ and $p_j$ and any other points we find lie in the same cell in both griddings.

We begin by supposing that $p_h$, $p_i$ and $p_j$ lie in a non-monotone cell in the first gridding. Up to grid mappings, we have a situation such as the one depicted in Figure~\ref{fig-antichain-end-cell} (note that other arrangements are possible). As $p_h$, $p_i$ and $p_j$ cannot form a monotone sequence, they also lie in a non-monotone cell in the second gridding. For any such set of three points, observe that one of $(p_h, p_i)$ or $(p_i, p_j)$ are a pair of consecutive pins (i.e.\ $h=i-1$ or $i=j-1$) and hence adjacent by position or by value --- without loss we will suppose $i=j-1$ and that $p_i$ and $p_j$ are adjacent by position. Note that this also excludes the possibility that either of $p_i$ or $p_j$ is $p_1$ or $p_n$, and so we may apply Lemma~\ref{lem-pin-pairs}. This yields a pair $(p_{i-1},p_{j+1})$, which in the first gridding lie in the only other non-empty cell in this row. In the second gridding, $(p_{i-1},p_{j+1})$ must also lie together and in a different cell from the one containing $p_h$, $p_i$ and $p_j$, to avoid creating a forbidden pattern. We now apply Lemma~\ref{lem-pin-pairs} again, this time to the pair $(p_{i-1},p_{j+1})$ which are adjacent by value, yielding another pair $(p_{i-2},p_{j+2})$ adjacent by position. In both griddings, $(p_{i-2},p_{j+2})$ must occupy the only other non-empty cell in the same column as $(p_{i-1},p_{j+1})$. Repeating this process, we follow points around both griddings until we reach the other non-monotone cell. At this point, all the pins we have listed so far are forced to lie in the same cells in both griddings.

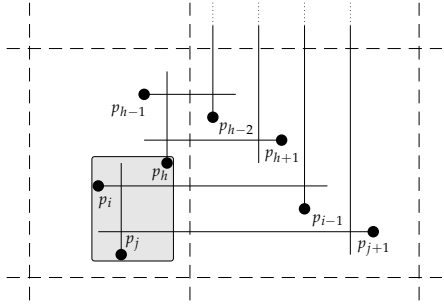
\begin{figure}
\begin{center}
\psset{xunit=0.012in, yunit=0.012in}
\psset{linewidth=0.005in}
\begin{pspicture}(0,0)(190,130)
{\psset{fillstyle=solid,cornersize=absolute,linearc=1pt,linewidth=0.005in,linestyle=solid,linecolor=black}
  \psframe[fillcolor=lightgray](37,17)(73,63)}
\pscircle*(40,50){0.03in}
\pscircle*(50,20){0.03in}
\pscircle*(60,90){0.03in}
\pscircle*(70,60){0.03in}
\pscircle*(90,80){0.03in}
\pscircle*(120,70){0.03in}
\pscircle*(130,40){0.03in}
\pscircle*(160,30){0.03in}
{\tiny
\rput[c](43,43){$p_{i}$}\rput[c](55,25){$p_{j}$}
\rput[c](53,83){$p_{h-1}$}\rput[c](67,55){$p_{h}$}
\rput[c](100,75){$p_{h-2}$}\rput[c](120,63){$p_{h+1}$}
\rput[c](140,35){$p_{i-1}$}\rput[c](160,23){$p_{j+1}$}}
\psline[linestyle=dashed](10,0)(10,130)
\psline[linestyle=dashed](80,0)(80,130)
\psline[linestyle=dashed](180,0)(180,130)
\psline[linestyle=dashed](0,10)(190,10)
\psline[linestyle=dashed](0,110)(190,110)
\psline(90,80)(90,120)
\psline(60,90)(100,90)
\psline(70,60)(70,100)
\psline(120,70)(60,70)
\psline(110,60)(110,120)
\psline(130,40)(130,120)
\psline(140,50)(40,50)
\psline(50,60)(50,20)
\psline(40,30)(160,30)
\psline(150,20)(150,120)
\psline[linestyle=dotted,dotsep=1pt](90,120)(90,130)
\psline[linestyle=dotted,dotsep=1pt](110,120)(110,130)
\psline[linestyle=dotted,dotsep=1pt](130,120)(130,130)
\psline[linestyle=dotted,dotsep=1pt](150,120)(150,130)
\end{pspicture}
\end{center}
\caption{Finding points adjacent by position or by value in the end cell for the proof of Lemma~\ref{lem-antichain-uniquely-griddable}.}
\label{fig-antichain-end-cell}
\end{figure}

In the case where $p_h$, $p_i$ and $p_j$ lie in a monotone cell in the first gridding, we observe that $p_h$ and $p_i$ are adjacent by position or value, and $p_i$ and $p_j$ are, respectively, adjacent by value or position. We now repeatedly apply Lemma~\ref{lem-pin-pairs} to both of these sets of pairs as before. In the first gridding, one pair will produce a sequence that finishes at one of the non-monotone cells, and the other pair will generate a sequence to reach the other. In the second gridding, the positions of all of these points must be the same: first, the cell containing $p_h$, $p_i$ and $p_j$ is monotone, as otherwise it is the unique cell in its row or column, and so must also contain both points from one of the pairs $(p_{h-1},p_{i+1})$ or $(p_{i-1},p_{j+1})$ which would produce a forbidden pattern. Thus this cell is monotone, and so the pairs $(p_{h-1},p_{i+1})$ and $(p_{i-1},p_{j+1})$ must (in some order) lie in the only other non-empty cells in its row and column. A similar argument applies to the cells containing each pair in turn, and hence each of these pairs lies in the same cell in both griddings.

In either of the above cases for the position of $p_h$, $p_i$ and $p_j$, the cells of both griddings now line up, and we have a sequence of pairs of points connecting one non-monotone cell to the other. All that remains is to show how to ``turn around'' in the non-monotone cells, as we can then follow sequences of pairs of points between the two end cells until we have considered every point.
Suppose, therefore, that we have reached a pair of non-consecutive points labelled $p_k$ and $p_\ell$ with $k<\ell-1$ in a non-monotone cell. Without loss we assume that $p_k$ and $p_\ell$ are adjacent by position, so that in both griddings the non-monotone cell is the only non-empty cell in its row, and the previous pair $(p_{k+1},p_{\ell-1})$ lie together in the only other nonempty cell in its column. Unless $k=1$ or $\ell=n$ (which we will consider shortly), we can apply Lemma~\ref{lem-pin-pairs} again to find pins $p_{k-1}$ and $p_{\ell+1}$ separating $p_k$ and $p_\ell$ by value. Reading from left to right, they appear in the order $p_{k-1}p_{k}p_{\ell}p_{\ell+1}$ or its reverse, so all four points must lie in the same cell in both griddings. We now have two pairs of points which are adjacent by value, namely $(p_{k-1},p_k)$ and $(p_\ell,p_{\ell+1})$, to which we can again repeatedly apply Lemma~\ref{lem-pin-pairs}. We follow both of these sequences until we reach the other non-monotone cell, at which point we can again ``turn around''. Note, however, that we only need to follow pairs which give rise to pins that have not yet been seen.

Finally, when we encounter a pair $(p_k,p_\ell)$ in a non-monotone cell with $k=1$ (or, by a similar argument, $\ell=n$), we find that they are separated by exactly one point, namely $p_{\ell+1}$. Now the pair $(p_\ell, p_{\ell+1})$ is adjacent by position or value, so we can use Lemma~\ref{lem-pin-pairs} on it as before. This then generates a sequence which we can follow back and forth until we reach the pin $p_n$, whence we have covered every point of the sequence. 

Thus all points $p_1,\ldots,p_n$ must be placed in the same cells in both griddings, as required. The extension to $f(\alpha)$ is trivial: when we encountered pin $p_1$ or $p_n$ in the above argument, we now encounter two points, both of which have their cell placements forced.
\end{proof}

\begin{lemma}\label{lem-ak-antichain}The set of permutations of length at least $2(k+1)^2+3$ in $A^k$ is an antichain with respect to permutation containment.\end{lemma}

\begin{proof}
Let $\alpha,\beta$ be two permutations in $A^k$ of lengths $m$ and $n$ respectively, both of length at least $2(k+1)^2+3$. Assuming $m<n$, suppose for a contradiction that $\alpha\leq\beta$, and fix one such embedding. Since both $\alpha$ and $\beta$ have unique $\M^k$-griddings, this implies not only that $\alpha\leq_{\M^k}\beta$, but that our fixed embedding witnesses this gridded containment. By their construction, we can write $\alpha$ and $\beta$ as inflated grid pin permutations, thus $\alpha=p_1[21],p_2,\ldots,p_m[\gamma]$ and $\beta=q_1[21],q_2,\ldots,q_n[\gamma]$, where $\gamma=12$ or $21$ depending on the parity of $k$. Moreover, since the griddings must match up, the fictive pin $p_0$ is placed in exactly the same position as $q_0$, and so we will assume that $p_0$ is mapped to $q_0$.

We claim that the inflated pin $p_1[21]$ must be mapped to $q_1[21]$. If not, then $p_1[21]$ must be mapped to two consecutive pins in the same cell, this being the only other way to form a 21 pattern in the cell labelled by $\C$. Thus suppose $p_1[21]$ is mapped to the pins $q_{2ki}$ and $q_{2ki+1}$. Then the left pin $p_2$ must be mapped to some left pin with index at most $2kj+2$, where $j<i$, since all later pins do not separate $q_{2ki}$ from $q_0$. Next, by a similar argument, $p_3$ can be mapped to a pin in $\beta$ with index at most $2kj+3$, and so on, until we find that pin $p_{2k}$ (which is the next pin of $\alpha$ we encounter in the cell labelled by $\C$) must be mapped to a pin with index at most $2kj+2k\leq 2ki$. This, however, is impossible because $p_{2k}$ must lie below and to the left of $p_1[21]$, but yet in $\beta$ its image cannot.

Now, since $p_0$ and $q_0$, and $p_1[21]$ and $q_1[21]$ coincide, by the properties of grid pin sequences we conclude that $p_2$ must be mapped to $q_2$, $p_3$ to $q_3$, and so on. This, however, becomes impossible when we try to map $p_m[\gamma]$ into $q_m$: by non-interaction, there are no pins other than $q_m$ in $\beta$ that separate $\rect(q_{m-2},q_{m-1})$, but yet we need two in $\alpha$ to separate $\rect(p_{m-2},p_{m-1})$.
\end{proof}

Thus we have:

\begin{proof}[Proof of Theorem~\ref{thm-not-pwo}] First note that if $G_\M$ contains a cycle then it contains a non-partially well-ordered class by Theorem~\ref{thm-murphy-vatter}, so we may assume that $G_\M$ is acyclic and has a component with at least two non-monotone griddable entries. By Lemma~\ref{lem-down-to-f}, it suffices to prove that $\grid(\M)$ is not partially well-ordered for a gridding matrix $\M$ belonging to the family of matrices $\F$, since every other grid class that needs to be considered contains such a class. 

By Lemmas~\ref{lem-antichain-uniquely-griddable} and~\ref{lem-ak-antichain}, $A^k$ contains an infinite antichain of strongly uniquely $\M^k$-griddable permutations. By Theorem~\ref{thm-matrix-symmetry}, $\M$ can be obtained from one of the matrices $\M^k$ for some $k$ via a grid mapping, and so $\grid(\M)$ is not partially well-ordered by Theorem~\ref{thm-strongly-uniquely-not-pwo}.
\end{proof}

The proof of Theorem~\ref{thm-fin-simple-grid} now follows by combining Theorems~\ref{thm-not-pwo} and~\ref{thm-grid-pwo}.

\section{Concluding Remarks}\label{sec-conclusion}

\paragraph{Monotone griddable classes.} Theorem~\ref{thm-fin-simple-grid} cannot immediately be extended by replacing ``monotone classes'' with ``monotone gridabble classes''. For example, if $\M=\left( \C\hspace{6pt}\D\right)$ where $\displaystyle\C=\binom{\av{12}}{\av{21}}$ and $\displaystyle\D=\binom{\av{21}}{\av{12}}$, then $\grid(\M)$ contains $\displaystyle\grid \binom{\av{12}\hspace{6pt}\av{21}}{\av{21}\hspace{6pt}\av{12}}$ which is not partially well-ordered by Theorem~\ref{thm-murphy-vatter}. (Note also that both $\C$ and $\D$ contain only finitely many simple permutations, so adding this restriction would not help.) However, Theorem~\ref{thm-fin-simple-grid} can be used indirectly for such gridding matrices by refining the gridding until all cells are monotone or non-monotone griddable --- the details of such a refinement are beyond the scope of this paper, but see Vatter~\cite{vatter:small-permutat:} for more details on griddability.

\paragraph{Grid pin sequences and antichains.} Currently, every known infinite antichain in the permutation containment order can be built, via grid symmetries, from an infinite grid pin sequence. Note, however, that not every known infinite antichain is constructed simply by inflating the first and last points of a grid pin sequence: see Murphy's thesis~\cite{murphy:restricted-perm:} for some other ``anchoring'' constructions. This naturally leads one to wonder whether there are infinite antichains that cannot be formed in this way. The \emph{closure} of a set $A$ of permutations is the class of permutations contained in the permutations of $A$, $\cl{A}=\{\pi : \pi\leq \alpha\textrm{ for some }\alpha\in A\}$.

\begin{question} Does there exist an infinite antichain $A$ for which $\cl{A}$ does not contain arbitrarily long grid pin sequences?\end{question}

\paragraph{Partial well-order decidability.} A first step in answering more general questions of decidability could be to consider the following question.

\begin{question}Is it decidable whether a given gridding matrix whose entries are partially well-ordered permutation classes defines a grid class that is partially well-ordered or not?\end{question}

Theorems~\ref{thm-fin-simple-grid} and~\ref{thm-not-pwo} make some progress towards answering this, particularly in the extension to monotone griddable classes discussed earlier. However, a complete answer would also need to consider gridding matrices where each component is a tree with entries given by monotone classes except for one cell, which is labelled by a non-monotone-griddable class with arbitrarily long simple permutations. This situation is currently amenable neither to Higman's Theorem nor grid pin sequences.

\paragraph{Acknowledgements.} The author wishes to thank Vince Vatter and Michael Albert for helpful discussions, and to the two referees whose comments greatly improved the presentation of these results.



\bibliographystyle{acm}
\bibliography{../refs}

\end{document}